\documentclass[reqno]{amsart}
\usepackage{amssymb,latexsym,amsmath,amsfonts,enumerate}
\usepackage{amsthm}
\usepackage[mathscr]{eucal}
\usepackage{paralist}
\usepackage{rotating}
\usepackage{multirow}
\usepackage[left=1.5in, right=1.5in]{geometry}
\newtheorem{theorem}{Theorem}
\newtheorem{lemma}{Lemma}
\newtheorem{proposition}{Proposition}
\newtheorem{corollary}{Corollary}

\theoremstyle{definition}
\newtheorem{definition}[theorem]{Definition}

\newtheorem{example}[theorem]{Example}
\theoremstyle{remark}
\newtheorem{remark}{Remark}

 \usepackage[colorlinks=true]{hyperref}
 \hypersetup{urlcolor=blue, citecolor=red}
\title[Superlinear Volterra Equations]{Blow--up and Superexponential Growth in Superlinear Volterra Equations}
\author[John A. D. Appleby and Denis D. Patterson]{}
\subjclass[2010]{Primary: 34K12, 34K25.}
\keywords{Volterra equations, blow-up solutions, growth, asymptotics, superlinear}

\email{john.appleby@dcu.ie}
\email{denis.patterson2@mail.dcu.ie}
\begin{document}
		\maketitle
	\centerline{\scshape John A. D. Appleby}
	\medskip
	{\footnotesize
		\centerline{School of Mathematical Sciences, Dublin City University}
		\centerline{Dublin, Ireland}
	} 
	
	\medskip
	
	\centerline{\scshape Denis D. Patterson}
	\medskip
	{\footnotesize
		\centerline{School of Mathematical Sciences, Dublin City University}
		\centerline{Dublin, Ireland}
	}
	
	\bigskip
	
		\begin{abstract}
			This paper concerns the finite--time blow--up and asymptotic behaviour of solutions to nonlinear Volterra integro--differential equations. Our main contribution is to determine sharp estimates on the growth rates of both explosive and nonexplosive solutions for a class of equations with nonsingular kernels under weak hypotheses on the nonlinearity. In this superlinear setting we must be content with estimates of the form $\lim_{t\to\tau}A(x(t),t) = 1$, where $\tau$ is the blow--up time if solutions are explosive or $\tau = \infty$ if solutions are global. Our estimates improve on the sharpness of results in the literature and we also recover well--known blow--up criteria via new methods.
		\end{abstract} 
	\section{Introduction}
	This paper concerns the blow--up and asymptotic behaviour of positive solutions to initial value problems of the form
	\begin{equation}\label{eq.volterra_ch6}
	x'(t) = \int_0^t w(t-s)f(x(s))\,ds, \quad t \geq 0; \quad x(0) > 0.
	\end{equation}
	We assume that the nonlinearity, $f$, obeys
	\begin{equation}\label{eq.f1_ch6}
	f \in C((0,\infty);(0,\infty)), \quad f \mbox{ is asymptotically increasing}, \quad \lim_{x\to\infty}\frac{f(x)}{x}=\infty.
	\end{equation}
	The positivity and monotonicity hypotheses in \eqref{eq.f1_ch6} are natural when studying growing solutions to \eqref{eq.volterra_ch6}. Moreover, $f(x)/x \to \infty$ as $x\to\infty$ is necessary for the existence of a solution to \eqref{eq.volterra_ch6} which blows up in finite--time. Sufficient conditions for the existence and uniqueness of local solutions are readily available \cite{GLS}. It is well--known that the behaviour of the kernel near zero is crucial in the analysis of blow--up problems of the type studied in this paper \cite{brunner2012blow}. Hence we assume that
	\begin{equation}\label{eq.w}
	w(0)>0, \quad w \in C(\mathbb{R}^+;\mathbb{R}^+) .
	\end{equation}
	There is a rich and active literature on blow--up problems in Volterra integral equations (VIEs) (see the survey articles \cite{roberts1998analysis,roberts2007recent} and the recent papers \cite{kirk2013system,ma2011blow}). Much of this interest stems from the connections between VIEs and PDEs of parabolic--type in which the source term has a highly localised spacial dependence \cite{kirk2002blow,mydlarczyk2005blow,olmstead1996explosion}. In this context, a blow--up solution represents the scenario in which the energy entering the system via the source term outweighs the ability of the medium to dissipate this energy and a literal explosion occurs in the physical system . In many cases, the leading order behaviour in such models is governed by a nonlinear VIE of the form
	\begin{equation}\label{eq.volterra_integral}
	x(t) = x(0) + H(t) + \int_0^t W(t-s)f(x(s))\,ds, \quad t \geq 0.
	\end{equation}
	Equation \eqref{eq.volterra_ch6} is a special case of \eqref{eq.volterra_integral}. In particular, if $W \in C^1([0,\infty);[0,\infty))$ with $W(0)=0$ and $H \equiv 0$, differentiation of \eqref{eq.volterra_integral} yields \eqref{eq.volterra_ch6} with $w=W'$. Similarly, integration of \eqref{eq.volterra_ch6} yields \eqref{eq.volterra_integral} with $H \equiv 0$. After analysing the unforced equation \eqref{eq.volterra_ch6}, we later extend our results to the case of nontrivial $H$ (see Section \ref{sec.ext_to_pert}). 
	
	According to the survey of Roberts \cite{roberts1998analysis}, research on blow--up problems of the type discussed above has mainly sought to answer the following questions:
	\begin{enumerate}[(1.)]
		\item Under what conditions do solutions blow--up?
		\item At what time do solutions blow--up?
		\item What is the asymptotic behaviour of solutions at blow--up?
	\end{enumerate}
	Being the most fundamental, $(1.)$ has naturally attracted the most attention and thus blow--up criteria for both general and specialised classes of VIEs are very well understood (see, for example, \cite{malolepszy2008blow}). We revisit $(1.)$ for the Volterra integro--differential equation (VIDE) \eqref{eq.volterra_ch6} and prove necessary and sufficient conditions for finite--time explosion of solutions. However, as we explain in more detail in Section \ref{sec_main_results_blow_up}, our conditions can be recovered from existing general conditions for equation \eqref{eq.volterra_integral} due to Brunner and Yang \cite{brunner2012blow}. We still find it useful to independently prove our own blow--up criteria for \eqref{eq.volterra_ch6} in order to gain preliminary insight into the behaviour of solutions. Moreover, our method of proof is different to that which Brunner and Yang used to tackle the related VIE problem. We do not address $(2.)$ -- estimation of the blow-up time -- in the present work, but this is also a very active area of investigation (see \cite{malolepszy2014blow,malolepszy2010blow} and the references therein) and represents an interesting open problem for general nonlinear VIEs.
	
	Our main contribution is to provide a comprehensive answer to $(3.)$ for equations of the form \eqref{eq.volterra_ch6}, and furthermore to understand the behaviour of nonexplosive solutions. The asymptotic behaviour of blow--up solutions has attracted considerable attention, both for VIEs and PDEs. Moreover, various authors have studied the problem of determining the blow--up rate or profile of solutions (see e.g. \cite{fu2003global,mahmoudi_2017}). Roberts \cite{roberts1997characterizing}, and Olmstead and Roberts \cite{roberts1996growth} study VIEs with parametric families of nonlinearities and kernels. They employ integral transform methods to estimate the growth rates of solutions but this work relies on conjecturing the leading order behaviour of solutions and finding a consistent ``asymptotic balance'' from the original equation, so the full proofs of these conclusions remains open; the results of this paper suggest that our methods may well play a role in resolving these open problems. Mydlarczyk provides very good estimates on the size of solutions to \eqref{eq.volterra_integral} in the presence of a blow-up with a power--type kernel \cite{MYDLARCZYK1994248,mydlarczyk1999blow}. However, these estimates do not give a sharp characterisation of the asymptotic growth rate of solutions. In particular, Mydlarczyk's results lead to conclusions of the form
	\[
	0 < \liminf_{t \to T^-}A(x(t),t) < \limsup_{t\to T^-}A(x(t),t) < \infty,
	\] 
	where $A$ is an appropriately chosen monotone function and $T$ is the blow--up time. Evtukhov and Samoilenko also study the power kernel case but specialise to regularly varying nonlinearities, in fact their particular interest is $n$--th order equations \cite{evtukhov2011asymptotic}. In this special case, they improve upon Mydlarczyk's results by proving that
	\[
	\lim_{t \to \tau}B_\tau(x(t),t) = 1, \quad \tau \in \{T, \infty\},
	\]
	for an appropriately chosen function $B_\tau$. To the best of our knowledge, this is the most complete result available in the extant literature.
	
	We first outline our results for the case $H \equiv 0$. Under \eqref{eq.f1_ch6} and \eqref{eq.w}, we identify a \emph{decreasing} function $F_B$ such that 
	\begin{equation}\label{eq.intro_comment_blowup}
	\lim_{t \to T^-}\frac{F_B(x(t))}{T-t} = \sqrt{2w(0)},
	\end{equation}
	where $T$ is the blow--up time. Similarly, in the nonexplosive case, we identify an \emph{increasing} function $F_U$ such that
	\begin{equation}\label{eq.intro_rate_comment}
	\lim_{t \to \infty}\frac{F_U(x(t))}{t} = \sqrt{2w(0)},
	\end{equation}
	under the additional assumption that $w \in L^1(\mathbb{R}^+;\mathbb{R}^+)$. The functions $F_B$ and $F_U$ depend only on $f$ and hence can be estimated from the problem data. Furthermore, our assumptions on the nonlinearity are nonparametric and allow a good deal of generality while still yielding strong conclusions. Interestingly, in spite of the dependence of these growth rates on $w$, the presence of a blow--up is completely \emph{independent} of the value of $w(0)$ and the structure of the kernel under \eqref{eq.w}. 
	
	If $H \in C^1([0,\infty);[0,\infty))$, then \eqref{eq.intro_comment_blowup} is unchanged. However, in the nonexplosive case, $H$ can impact the growth rate of solutions. When $H$ is sufficiently small the growth rate from \eqref{eq.intro_rate_comment} is preserved and we characterise these rate preserving perturbations.
	
	The outline of the paper is as follows: in  Section \ref{sec_main_results_blow_up} we give precise blow--up criteria for equation \eqref{eq.volterra_ch6}, explain how they can be recovered from previous work, and outline the novelty of our methods. Section \ref{sec_growth_rates} details the asymptotic growth rates of solutions to \eqref{eq.volterra_ch6} when $H \equiv 0$ and Section \ref{sec.ext_to_pert} extends these results to the case when $H$ is nontrivial. We provide some simple examples to illustrate the application of our results in Section \ref{sec_examples}. All proofs are deferred to the closing sections of the paper; Section \ref{proofs_prelim} contains proofs of preliminary results and lemmas while Section \ref{proofs_main} contains the proofs of our main results.
	\section{Blow--up Conditions}\label{sec_main_results_blow_up}
	\begin{definition}\label{defn.blow_up}
		A solution to \eqref{eq.volterra_ch6} blows up in finite--time if there exists $T>0$ such that $x \in C([0,T);[0,\infty))$ but $\lim_{t\to T^-}|x(t)| = \infty$; the minimal such $T$ is the blow--up time.
	\end{definition} 
	The following result characterises the finite--time blow--up of solutions to \eqref{eq.volterra_ch6}.
	\begin{theorem}\label{thm.sufficient_blow_up}
		Suppose \eqref{eq.f1_ch6} and \eqref{eq.w} hold. Solutions to \eqref{eq.volterra_ch6} blow--up in finite--time if and only if \begin{equation}\label{condition.finite}
		\int_{\eta}^\infty \frac{du}{\sqrt{\int_0^u f(s)\,ds}} < \infty,\quad \mbox{ for some }\eta>0.
		\end{equation}
	\end{theorem}
	Under \eqref{eq.f1_ch6}, the negation of \eqref{condition.finite} is of course
	\begin{equation}\label{condition.infinite}
	\int_{\eta}^\infty \frac{du}{\sqrt{\int_0^u f(s)\,ds}} = \infty, \quad\mbox{ for all }\eta>0,
	\end{equation}
	and, by Theorem \ref{thm.sufficient_blow_up}, condition \eqref{condition.infinite} guarantees that solutions to \eqref{eq.volterra_ch6} are global; we record condition \eqref{condition.infinite} for future reference.
	
	Theorem \ref{thm.sufficient_blow_up} is a special case of the following result.
	\begin{theorem}[Brunner and Yang {\cite[Theorem 3.9]{brunner2012blow}}]\label{brunner}
		Suppose $\psi>0$, $h(t) \geq 0$ for $t\geq 0$, $w(t) = t^{\beta-1}w_1(t) \geq 0$ for $t\geq 0$, $\beta>0$, and $w_1$ is bounded on every compact interval with $\inf_{s\in [0,\delta]}w_1(s)>0$ for some $\delta>0$. Suppose that $G:\mathbb{R}^+ \times \mathbb{R}^+ \mapsto \mathbb{R}^+$ is continuous (uniformly in its second argument), increasing in its second argument, and satisfies $\lim_{u\to\infty}G(0,u)/u=\infty$. Solutions to 
		\begin{align}\label{hammerstein}
		u'(t) = h(t) + \int_0^t w(t-s)G(s,u(s))\,ds, \quad t \geq 0,
		\end{align}
		blow--up in finite--time if and only if there exists a $t^*>0$ such that
		\begin{align}\label{eq.t_star}
		\int_0^{t^*} h(s)\,ds + \min_{u \in [0,\infty)}\left( \int_0^{t^*} W(t^*-s)G(s,u)\,ds - u \right) >0, \quad W(t)=\int_0^t w(s)\,ds,
		\end{align}
		and
		\begin{align}\label{eq.BY_condition}
		\int_{\eta}^\infty \left(\frac{u}{G(t^*,u)}\right)^{1/(1+\beta)} \frac{du}{u} < \infty, \quad\mbox{for some }\eta>0.
		\end{align}
	\end{theorem}
	To recover Theorem \ref{thm.sufficient_blow_up} from Theorem \ref{brunner}, set $h \equiv 0$, $\beta = 1$, and $G(s,u) = f(u)$. Thus \eqref{eq.t_star} holds if $\min_{u \in [0,\infty)}\left(f(u) \int_0^{t^*}W(s)\,ds - u\right)>0$ and we can choose $t^*>0$ sufficiently large to satisfy this condition since $\int_0^\infty W(s)\,ds=\infty$. In our case, condition \eqref{eq.BY_condition} reduces to the finiteness of the integral 
$
	\int_1^\infty dx/\sqrt{xf(x)}
$, but if $f\in C((0,\infty);(0,\infty))$ is increasing, then
		\begin{align}\label{condition.ours}
		\int_1^\infty \frac{dx}{\sqrt{xf(x)}} < \infty\quad \mbox{ if and only if }\quad\int_1^\infty \frac{dx}{\sqrt{\int_0^x f(s)\,ds}} < \infty,
		\end{align}
		i.e. the conclusions of Theorems \ref{thm.sufficient_blow_up} and \ref{brunner} are consistent.
		
	While the conclusion of Theorem \ref{thm.sufficient_blow_up} is known, unlike Theorem \ref{brunner}, its proof yields considerable insight into the rate at which solutions to \eqref{eq.volterra_ch6} grow. The proof of Theorem \ref{brunner} proceeds by integrating \eqref{hammerstein} to obtain an integral equation of the form
	\[
	u(t) = u(0) + H(t) + \int_0^t W(t-s) G(s,u(s))\,ds, \quad t \geq 0. 
	\] 
	The integral equation above is discretised along a sequence $(t_n)_{n \geq 1}$ upon which the solution to \eqref{hammerstein} grows geometrically, i.e. $u(t_n) = R^n$ for each $n\geq 1$ and some $R>1$. In all cases, $\lim_{n\to\infty}{t_{n+1} - t_n} = 0$ and moreover, if there is a global solution, $h_n = t_{n+1}-t_n$ tends to zero so fast that $\sum_{n=1}^\infty h_n < \infty$, contradicting the existence of a global solution. Conversely, in the presence of a blow--up solution, $(h_n)_{n \geq 1}$ is proven not to be summable using similar difference inequalities. Hence $\lim_{n\to\infty}t_n=\infty$, contradicting the assumption that the solution explodes in finite--time. In both cases, the summability of the sequence $(h_n)_{n \geq 1}$ hinges on \eqref{eq.BY_condition}. Naturally, some rough rates of growth are implicit in the construction described above, but it is difficult to see how one could obtain sharp estimates on rates of asymptotic growth of solutions from this approach, even for the simpler equation \eqref{eq.volterra_ch6}.
	
	In contrast, we exploit the enhanced differential structure of \eqref{eq.volterra_ch6} and employ comparison equations of the form
	\begin{align}\label{eq.bounded_discuss}
	z'(t) &= C \int_{t-\delta}^t f(z(s))\,ds, \quad t \geq T^* \geq 0, \quad\mbox{ with } \delta>0 \mbox{ and }C >0,
	\end{align}
	to establish sharp blow--up conditions. The fact that comparison equations such as \eqref{eq.bounded_discuss} yield sharp blow--up criteria suggests that these bounded delay equations are promising candidates for investigating the more subtle issue of asymptotic behaviour. Under mild continuity assumptions,
	\begin{align}\label{eq.2nd_order_discuss}
	z''(t) = C f(z(t)) - C f(z(t-\delta)), \quad t > T^* > \delta.
	\end{align}
	Solutions of \eqref{eq.volterra_ch6} and \eqref{eq.bounded_discuss} will grow extremely rapidly when $f(x)/x \to \infty$ as $x\to\infty$ so we conjecture that the delayed term in \eqref{eq.2nd_order_discuss} is negligible asymptotically. Following this line of reasoning, we expect the second order ODE 
	$
	z''(t) = f(z(t))
	$
	to give a good asymptotic approximation to solutions of \eqref{eq.volterra_ch6}; this approximation is at the heart of our analysis and the definitions which follow are the product of our efforts to systematically exploit this idea.
	\begin{definition}\label{defn.super_exp}
		We say $g \in C((0,\infty);(0,\infty))$ exhibits \emph{superexponential growth} if $g(x)\to\infty$ as $x\to\infty$ and \[
		\lim_{x\to\infty}\frac{g(x-\epsilon)}{g(x)} = 0, \quad \mbox{for each }\epsilon > 0.
		\]
	\end{definition}
	Continuous, positive functions which obey $g'(x)/g(x) \to \infty$ as $x\to\infty$ exhibit superexponential growth and this motivates our choice of terminology.
	\begin{definition}\label{def.preserves_super_exp}
		$\phi \in C((0,\infty);(0,\infty))$ \emph{preserves superexponential growth} if for each function $g$ which exhibits superexponential growth and each $\epsilon>0$, we have
		\[
		\lim_{x\to\infty}\frac{\phi(g(x-\epsilon))}{\phi(g(x))} = 0.
		\]
	\end{definition}
	\noindent If $\phi,f\in C((0,\infty);(0,\infty))$ obey $\phi \sim f$ and $\phi$ preserves superexponential growth, then so does $f$. The following lemma (whose proof is elementary and thus omitted) records several important classes of nonlinear functions which preserve superexponential growth and frequently arise in applications.
	\begin{proposition}\label{LEMMA.PRESERVES}
		If $\phi \in  C([0,\infty) ; [0,\infty))$ obeys any of the following conditions
		\begin{enumerate}[(i.)]
			\item $x \mapsto \phi(x)/x$ is eventually increasing,
			\item $\phi$ is increasing and convex,
			\item $\phi \in \text{RV}_\infty(\alpha)$ for some $\alpha>0$,
		\end{enumerate}
		then $\phi$ preserves superexponential growth.
	\end{proposition}
	\begin{remark}
	$f : x \mapsto f(x)$ is eventually increasing if exists a number $X$ such that $f(x)$ is increasing for $x \in [X,\infty)$.
	\end{remark}
	\section{Growth Rates of Solutions}\label{sec_growth_rates}
	In order to compute rates of growth of solutions, define the functions
	\begin{equation}\label{defn.F_blow_up}
	F_B(x) = \int_x^\infty \frac{du}{\sqrt{\int_0^u f(s)\,ds}}, \quad\mbox{ for each } x > 0,
	\end{equation}
	and 
	\begin{equation}\label{defn.F_unbounded}
	F_U(x) = \int_1^x \frac{du}{\sqrt{\int_0^u f(s)\,ds}}, \quad\mbox{ for each } x > 0.
	\end{equation}
	$F_B$ characterises the rate of growth to infinity of solutions which blow--up in finite time, while $F_U$ captures rates of growth of unbounded but nonexplosive solutions. In order to compute growth rates, we ask that the nonlinearity preserves superexponential growth, in the sense of Definition \ref{def.preserves_super_exp}. As discussed in Section \ref{sec_growth_rates}, preservation of superexponential growth is a relatively mild hypothesis satisfied by broad classes of nonlinearities commonly found in applications (see Proposition \ref{LEMMA.PRESERVES}).
	\begin{theorem}\label{thm.expl_rate}
		Suppose \eqref{eq.f1_ch6} and \eqref{eq.w} hold. If \eqref{condition.finite} holds and $f$ preserves superexponential growth, then solutions to \eqref{eq.volterra_ch6} blow--up in finite--time and obey
		\[
		\lim_{t\to T^-}\frac{F_B(x(t))}{T-t} = \sqrt{2 w(0)},
		\]
		where $T$ denotes the blow--up time.
	\end{theorem}
	When studying growth rates of non--explosive solutions, we further ask that
	\begin{equation}\label{eq.w_L^1}
	w \in L^1(\mathbb{R}^+;\mathbb{R}^+), \quad ||w||_{L^1} = {\mathcal W}.
	\end{equation}
	If $w$ does not have finite $L^1$--norm, then it can contribute to faster growth in the convolution term when the solution is global; assuming \eqref{eq.w_L^1} rules this out and allows us to prove the following analogue of Theorem \ref{thm.expl_rate} for non--explosive solutions.
	\begin{theorem}\label{thm.non_expl_rate}
		Suppose \eqref{eq.f1_ch6}, \eqref{eq.w}, and \eqref{eq.w_L^1} hold. If \eqref{condition.infinite} holds and $f$ preserves superexponential growth, then solutions to \eqref{eq.volterra_ch6} obey $x \in C([0,\infty);(0,\infty))$ and
		\begin{equation}\label{eq.unbounded_rate_unperturbed}
		\lim_{t\to\infty}\frac{F_U(x(t))}{t} = \sqrt{2 w(0)}.
		\end{equation}
	\end{theorem}
	The final result of this section shows that when $w(0) = 0$ and \eqref{condition.infinite} holds, solutions to \eqref{eq.volterra_ch6} do not  blow--up in finite--time. Furthermore, the rate of growth of solutions to \eqref{eq.volterra_ch6} must be strictly slower than the case when $w(0)>0$. More precisely, we assume
	\begin{align}\label{eq.w2}
	w \in C( [0,\infty);[0,\infty)), \quad w(0)=0, \quad w(t)> 0 \mbox{ for } t \in (0,\delta] \mbox{ for some }\delta>0.
	\end{align}
	\begin{theorem}\label{thm.w(0)=0}
		Suppose \eqref{eq.f1_ch6}, \eqref{eq.w_L^1}, and \eqref{eq.w2} hold. If \eqref{condition.infinite} holds, solutions to \eqref{eq.volterra_ch6} obey $x \in C([0,\infty);(0,\infty))$. If $f$ also preserves superexponential growth, then
		\begin{align}\label{eq.rate_slower}
		\lim_{t\to\infty}\frac{F_U(x(t))}{t} = 0.
		\end{align}
	\end{theorem}
	The proof of Theorem \ref{thm.w(0)=0} is a minor variation on arguments used throughout this paper and is hence omitted.
	\section{Extensions to Perturbed Equations}\label{sec.ext_to_pert}
	We now consider the case when a nonautonomous forcing term is added to \eqref{eq.volterra_ch6}, i.e.
	\begin{equation}\label{eq.volterra_forced}
	x'(t) = h(t) + \int_0^t w(t-s)f(x(s))\,ds, \quad t \geq 0; \quad x(0) = \psi>0,
	\end{equation}
	and demonstrate that the results of Section \ref{sec_main_results_blow_up} are preserved under ``small'' perturbations. We do not require $h$ to be nonnegative and hence solutions to \eqref{eq.volterra_forced} are no longer necessarily monotone; due to the nature of our comparison arguments this relaxation does not present any additional difficulties. Suppose that the forcing term, $h$, obeys
	\begin{align}\label{eq.H}
	h \in C(\mathbb{R};\mathbb{R}), \quad H(t) := \int_0^t h(s)\,ds \geq 0\quad \mbox{ for each }t \geq 0.
	\end{align}
	Results regarding the finite--time blow--up of solutions require no additional hypotheses. However, for results regarding rates of growth we ask that the nonlinearity obeys
	\begin{equation}\label{eq.f2_ch6}
	f \in C((0,\infty);(0,\infty)), \quad f \mbox{ is increasing}, \quad \lim_{x\to\infty}\frac{f(x)}{x}=\infty,
	\end{equation}
	in order to simplify and shorten the proofs.
	
	Our first result regarding solutions to the forced Volterra equation \eqref{eq.volterra_forced} shows that the blow--up condition and rate of explosion are unchanged by forcing terms obeying \eqref{eq.H}.
	\begin{theorem}\label{thm.expl_rate_perturbed}
		Suppose \eqref{eq.f1_ch6}, \eqref{eq.w}, and \eqref{eq.H} hold. If \eqref{condition.finite} holds, then solutions to \eqref{eq.volterra_forced} blow--up in finite--time. If we further suppose that $f$ preserves superexponential growth and \eqref{eq.f2_ch6} holds, then solutions to \eqref{eq.volterra_forced} obey
		\begin{equation}\label{eq.blow_up_rate_preserved}
		\lim_{t\to T^-}\frac{F_B(x(t))}{T-t} = \sqrt{2 w(0)},
		\end{equation}
		where $T$ denotes the blow--up time.
	\end{theorem}
	Previously we assumed that $f$ preserves superexponential growth when proving results regarding the rate of growth of solutions; henceforth we replace this hypothesis with the assumption that 
	\begin{equation}\label{eq.eventually_increasing}
	x \mapsto f(x)/x \quad\mbox{is eventually increasing}.
	\end{equation} 
	By Proposition \ref{LEMMA.PRESERVES}, $f$ preserves superexponential growth when \eqref{eq.eventually_increasing} holds. As we show presently, the stronger hypothesis \eqref{eq.eventually_increasing} allows us to characterise the perturbation terms which preserve the rate of growth when $h \equiv 0$, i.e. the asymptotic relation \eqref{eq.unbounded_rate_unperturbed} still holds, in the non--explosive case. Our next result also shows that our blow--up conditions remain necessary if the nonautonomous forcing term is sufficiently small in an appropriate sense.
	\begin{theorem}\label{thm.non_expl_rate_perturbed}
		Suppose \eqref{eq.f1_ch6}, \eqref{eq.w}, and \eqref{eq.H} hold. If \eqref{condition.infinite} holds, then solutions to \eqref{eq.volterra_forced} obey $x \in C([0,\infty);(0,\infty))$. If we further suppose $x \mapsto f(x)/x$ is eventually increasing, \eqref{eq.f2_ch6} holds, and $w$ obeys \eqref{eq.w_L^1}, then the following are equivalent:
			\[	(i.)\quad	\limsup_{t\to\infty}\frac{F_U(H(t))}{t} \leq \sqrt{2 w(0)}, \quad (ii.)\quad
			\lim_{t\to\infty}\frac{F_U(x(t))}{t} = \sqrt{2 w(0)}.
			\]
	\end{theorem}
	It is evidently of interest to study the case when $\limsup_{t\to\infty}F_U(H(t))/t > \sqrt{2 w(0)}$ and we conjecture that the perturbation likely dominates the dynamics of the system in this case. The results of \cite{appleby2017growth} provide a road map as to how this issue could be addressed.
	\section{Examples}\label{sec_examples}
	Since our results are insensitive to the structure of the memory, the examples which follow do not require a functional form for $w$ (so long as continuity and integrability assumptions hold). For example, with $\omega > 0$ arbitrary, the following kernels would be admissible:
	\begin{align*}
	w_1(t) = \omega (1+t)^{-\alpha}, \,\, \alpha \geq 0; \quad 
	w_2(t) = \omega \exp(-t^\gamma),\,\,\gamma>0; \quad
	w_3(t) = \omega/\Gamma(t+1),\,\,\gamma>0.
	\end{align*}
	where $\Gamma$ denotes the Gamma function.
	\begin{example}\label{eg.blow_up_1}
		Suppose $f(x) = (1+x)^\beta$ for all $x>0$ and for some $\beta>1$. Choose any $w$ obeying \eqref{eq.w}. Note that this choice of $f$ obeys \eqref{eq.f1_ch6} and also preserves superexponential growth; to see this check any of $(i.-iii.)$ in Proposition \ref{LEMMA.PRESERVES}. We first check condition \eqref{condition.finite} to determine whether or not solutions to \eqref{eq.volterra_ch6} blow--up in finite--time. First note that
		\[
		\sqrt{\int_0^u f(s)\,ds} = \left( \int_0^u (1+s)^\beta \,ds \right)^{1/2} = \left( \frac{(u+1)^{\beta+1} - 1}{\beta+1} \right)^{1/2}, \quad u \geq 0. 
		\]
		For $\eta>0$ arbitrary and $N>0$ sufficiently large, we have
		\[
		\int_{\eta}^N \frac{du}{\sqrt{\int_0^u f(s)\,ds}} = \sqrt{\beta+1}\,\int_\eta^N \left( (u+1)^{\beta+1} - 1 \right)^{-1/2} \,du. 
		\]
		As $u \to \infty$, $\left( (u+1)^{\beta+1} - 1 \right)^{-1/2} \sim u^{-(\beta+1)/2}$ and \eqref{condition.finite} holds since
		\[
		\int_\eta^\infty u^{-(\beta+1)/2} \, du = \frac{2 \eta^{(1-\beta)/2}}{\beta-1} < \infty, \quad\mbox{ for each }\eta >0 \mbox{ and } \beta > 1.
		\]
		Therefore, by Theorem \ref{thm.sufficient_blow_up}, solutions to \eqref{eq.volterra_ch6} blow--up for every $w$ obeying \eqref{eq.w}. It can be shown that
		\[
		F_B(x) \sim \frac{2(\beta-1)}{\sqrt{\beta+1}}\, x^{(1-\beta)/2}, \quad\mbox{ as }x\to \infty.
		\]
		Thus, by Theorem \ref{thm.expl_rate}, solutions to \eqref{eq.volterra_ch6} obey
		\begin{equation}\label{eq.expls_example}
		\lim_{t \to T^-} \frac{x(t)^{(1-\beta)/2}}{T-t} = \frac{1}{\beta-1}\sqrt{\frac{(\beta+1)w(0)}{2}},\quad \beta>1, \quad w(0)>0.
		\end{equation}
		Furthermore, solutions to \eqref{eq.volterra_forced} will still obey \eqref{eq.expls_example} for any perturbation $h$ obeying \eqref{eq.H}.
		
		In this example one may ``invert'' the asymptotic relation \eqref{eq.expls_example} to obtain the leading order behaviour of the solution at blow--up. In other words, \eqref{eq.expls_example} can be improved to
		\[
		x(t) \sim \left( \frac{1}{\beta-1}\sqrt{\frac{(\beta+1)w(0)}{2}} \right)^{2/(1-\beta)}(T-t)^{2/(1-\beta)}, \quad\mbox{ as } t \to T^-.
		\]
	\end{example}
	\begin{example}\label{eg_nonexplosive_2}
		Suppose $f(x) = (x+e)\log(x+e)$ for $x>0$ and let $w$ obey \eqref{eq.w}. Once again, it is straightforward to verify that $f$ satisfies \eqref{eq.f1_ch6} and preserves superexponential growth. Moreover, $x \mapsto f(x)/x = (x+e)\log(x+e)/x$ is eventually increasing.
		
		We first check condition \eqref{condition.finite} to see if solutions to \eqref{eq.volterra_ch6} blow--up in finite--time. Direct computation shows that
		\[
		\int_\eta^N \frac{du}{\sqrt{\int_0^u f(s)\,ds}} = 2 \int_{\eta}^N \frac{du}{\sqrt{(u+e)^2 \left(2\log(u+e)-1\right) -e^2 } }, \quad N> \eta>0.
		\]
		As $u \to \infty$,
		\[
		\sqrt{(u+e)^2 \left(2\log(u+e)-1\right) -e^2 } \sim u\sqrt{2\log(u)}.
		\]
		Thus \eqref{condition.finite} does not hold because
		\[
		\int_\eta^N \frac{du}{u\sqrt{2\log(u)}} = \sqrt{2}\left( \sqrt{\log(N)} - \sqrt{\log(\eta)} \right) \to \infty, \quad\mbox{ as } N \to \infty.
		\]
		Therefore, by Theorem \ref{thm.sufficient_blow_up}, solutions to \eqref{eq.volterra_ch6} are global if $w$ obeys \eqref{eq.w}. Furthermore,
		\[
		F_U(x) \sim 2\sqrt{2 \log(x)}, \quad\mbox{ as }x \to \infty
		\]
		and thus, by Theorem \ref{thm.non_expl_rate}, solutions to \eqref{eq.volterra_ch6} obey
		\begin{equation}\label{eq.eg_growth_rate}
		\lim_{t\to \infty} \frac{\log(x(t))^{1/2}}{t} = \frac{\sqrt{w(0)}}{2}.
		\end{equation}
		Equation \eqref{eq.eg_growth_rate} is of course equivalent to saying that $\log(x(t)) \sim w(0)t^2 /4$ as $t\to\infty$.
		
		Now we consider the effect of forcing terms on the asymptotic growth rate captured by \eqref{eq.eg_growth_rate}. Firstly suppose $h$ obeys \eqref{eq.H} and $H(t) \sim t^\alpha$ as $t \to \infty$, for some $\alpha>0$. Then
		\[
		\limsup_{t\to\infty}\frac{F_U(H(t))}{t} = \limsup_{t\to\infty}\frac{2\sqrt{2 \log(t^\alpha)}}{t} = 0, \quad \alpha > 0.
		\] 
		Hence, by Theorem \ref{thm.non_expl_rate_perturbed}, solutions to \eqref{eq.volterra_forced} still obey \eqref{eq.eg_growth_rate} for any perturbation tending to infinity no faster than a power.
	\end{example}
	\section{Preliminary Results and Lemmas}\label{proofs_prelim}
	We first characterise the behaviour of solutions of two auxiliary equations, namely
	\begin{align}\label{eq.bounded_delay_kernel}
	y'(t) &= \int_{t-\delta}^t w(t-s) f(y(s))\,ds, \quad t \geq 0;\quad 
	y(t) = \psi(t), \quad t \in [-\delta,0],
	\end{align}
	and
	\begin{align}\label{eq.bounded_delay}
	z'(t) &= C\,\int_{t-\delta}^t f(z(s))\,ds, \quad t \geq 0;\quad 
	z(t) = \psi(t), \quad t \in [-\delta,0],
	\end{align}
	for some $C>0$ and $\delta>0$. We often use solutions to equations of the form \eqref{eq.bounded_delay_kernel} and \eqref{eq.bounded_delay} as comparison solutions for the more complex Volterra equations \eqref{eq.volterra_ch6} and \eqref{eq.volterra_forced}. The hypotheses on the nonlinearity are as before and the initial function, denoted by $\psi$, is assumed positive throughout, i.e.
	\begin{equation}\label{eq.psi}
	\psi \in C([-\delta,0];(0,\infty)).
	\end{equation}
	The function $\bar{F}$ given by 
	\begin{equation}\label{def.bigF}
	\bar{F}(x) = \int_{0}^x f(s)\,ds, \quad \mbox{for each }x > 0,
	\end{equation}
	appears frequently and inherits useful properties from $f$, as noted in the following corollary.
	\begin{corollary}\label{corollary.convex}
		If \eqref{eq.f1_ch6} holds, then $\bar{F}$ preserves superexponential growth.
	\end{corollary}
	Corollary \ref{corollary.convex} follows directly from Proposition \ref{LEMMA.PRESERVES} by noting that $\bar{F}$ is the integral of a positive and increasing function, and thus is both increasing and convex itself.
	\begin{lemma}\label{lemma.bounded_delay_1}
		Let $C>0$ and $\delta>0$, and suppose that \eqref{eq.f1_ch6} and \eqref{eq.psi} hold. If a solution to \eqref{eq.bounded_delay} obeys $z \in C([-\delta,\infty);(0,\infty))$, then $z$ exhibits superexponential growth.
	\end{lemma}
	\begin{proof}[Proof of Lemma \ref{lemma.bounded_delay_1}]
		Assuming $z \in C([-\delta,\infty);(0,\infty))$ and \eqref{eq.psi} implies that $t \mapsto z(t)$ is increasing for $t \in [0,\infty)$ and hence that $\lim_{t\to\infty}z(t)=\infty$. Suppose $\sigma \in (0,\delta]$; let $t > 2\delta$ and integrate \eqref{eq.bounded_delay} from $t-\sigma$ to $t$ to obtain
		\begin{align*}
		z(t) - z(t-\sigma) &= C\,\int_{t-\sigma}^t \int_{s-\delta}^s f(z(u)) \,du\,ds =  C\,\int_{t-\sigma - \delta}^t \int_{(t-\sigma)\vee u}^{t \wedge (u+\delta)} f(z(u)) \,ds\,du,
		\end{align*}
		for each $t > 2\delta$. Using the positivity of $z$ and \eqref{eq.f1_ch6} yields the lower bound
		\[
		z(t) - z(t-\sigma) \geq C\,\int_{t-\sigma}^t \int_{(t-\sigma)\vee u}^{t \wedge (u+\delta)} f(z(u)) \,ds\,du \geq C\,\int_{t-\sigma}^t (t-u) f(z(u)) \,du, 
		\] 
		for each $t>2\delta$. The estimate above can (equivalently) be written as
		\[
		\frac{z(t)}{z(t-\sigma)} \geq 1 + \frac{C\,\int_{t-\sigma}^t (t-u) f(z(u)) \,du}{z(t-\sigma)}, \quad \mbox{ for each }t>2\delta.
		\]
		By \eqref{eq.f1_ch6}, there exists a continuous, increasing function $\phi$ such that $\phi(x) \sim f(x)$ as $x \to \infty$. Since $\lim_{t\to\infty}z(t)=\infty$, for each $\epsilon \in (0,1)$, there exists $T(\epsilon)>0$ such that $f(z(t)) > (1-\epsilon)\phi(z(t))$ for all $t \geq T(\epsilon)$. Thus,	by making the substitution $\alpha = t-u$ and using the monotonicity of $\phi$, it can be shown that 
		\[
		\int_{t-\sigma}^t (t-u) f(z(u)) \,du = \int_0^\sigma \alpha f(z(t-\alpha))\,d\alpha > \frac{(1-\epsilon)\sigma^2}{2}\phi(z(t-\sigma)),
		\]
		for each $t >2\delta+T(\epsilon)+\sigma$.
		Hence 
		\[
		\frac{z(t)}{z(t-\sigma)}> 1+  \frac{C\,(1-\epsilon)\sigma^2}{2}\frac{\phi(z(t-\sigma))}{z(t-\sigma)}, \quad t > 2\delta+T(\epsilon)+\sigma.
		\]
		Since $f(x)/x \to \infty$ as $x\to\infty$ and $\lim_{t\to\infty}z(t-\sigma)=\infty$ for each $\sigma \in (0,\delta]$, taking the liminf in the inequality above shows that $\lim_{t\to\infty}z(t)/z(t-\sigma) = \infty$. Therefore
		\begin{equation}\label{lim.super_exp}
		\lim_{t\to\infty}\frac{z(t-\sigma)}{z(t)}=0, \quad\mbox{for each $\sigma \in (0,\delta]$}.
		\end{equation}
		Finally, since $z$ is monotonically increasing, $z(t-\delta) \geq z(t-\sigma)$ for each $\sigma>\delta$, and for $t$ sufficiently large. Hence, from \eqref{lim.super_exp}, \[
		0 = \limsup_{t\to\infty}\frac{z(t-\delta)}{z(t)} \geq \limsup_{t\to\infty}\frac{z(t-\sigma)}{z(t)} \geq 0, \quad\mbox{for each $\sigma > \delta$}.
		\]
		Thus \eqref{lim.super_exp} holds for all $\sigma>0$ and $z$ obeys Definition \ref{defn.super_exp}, as required. 
	\end{proof}
	We immediately have the following useful lemma which we record for future use. 
	\begin{lemma}\label{lemma.bounded_delay_2}
		Let $C>0$ and $\delta>0$, and suppose that \eqref{eq.f1_ch6} and \eqref{eq.psi} hold. If the solution to \eqref{eq.bounded_delay} obeys $z \in C([-\delta,\infty);(0,\infty))$ and $\bar{F}$ is defined by \eqref{def.bigF}, then
		\[
		\lim_{t\to\infty}\bar{F}(z(t-\delta))/\bar{F}(z(t))=0.
		\] 
	\end{lemma}
	\begin{lemma}\label{lemma.bounded_delay_kernel_explodes}
		Let $\delta>0$, and suppose that \eqref{eq.f1_ch6}, \eqref{eq.w}, and \eqref{eq.psi} hold. If \eqref{condition.finite} holds, then solutions to \eqref{eq.bounded_delay_kernel} blows up in finite--time.
	\end{lemma}
	\begin{proof}[Proof of Lemma \ref{lemma.bounded_delay_kernel_explodes}]
		Under the stated hypotheses there is a continuous solution to \eqref{eq.bounded_delay_kernel} on an interval $[-\delta,T)$ for some $T>0$. Suppose $T=\infty$, let $t\geq\delta$, and estimate as follows:
		\begin{align*}
		y'(t) &= \int_{t-\delta}^t w(t-s) f(y(s))\,ds = \int_0^\delta w(u) f(y(t-u))\,du \geq \inf_{s \in [0,\delta]}w(s) \int_{0}^\delta f(y(t-u))\,du. 
		\end{align*}
		Define $\underbar w(\delta)=\inf_{s \in [0,\delta]}w(s)$ and note that \eqref{eq.w} guarantees $\underbar w(\delta)>0$. Hence 
		$
		y'(t) > \underbar w(\delta) \int_{t-\delta}^t f(y(s))\,ds$ for $t \geq \delta.
		$
		By \eqref{eq.f1_ch6}, there exists a continuous, increasing function $\phi$ such that for each $\epsilon \in (0,1/2)$, $f(y(u)) > (1-\epsilon)\phi(y(u))$ for each $u \geq T_1(\epsilon)+\delta$. Define $z$ by
		\begin{align}\label{def.lower_comparison}
		z'(t) = \underbar{w}(\delta)(1-2\epsilon)\int_{t-\delta}^t \phi(z(s))\,ds, \quad t \geq T_1(\epsilon) + \delta; \quad z(t)=\frac{y(t)}{2},
		\end{align}
		for $t \in [0,T_1(\epsilon) + \delta]$. By construction, $z(t)<y(t)$ for each $t \in [0,T_1+\delta]$. Hence, by a simple time of the first breakdown argument, $z(t)< y(t)$ for all $t\geq 0$. Due to the continuity of $\phi$, $z \in C^2((T_1+\delta,\infty);(0,\infty))$ and because $\phi \circ z$ is increasing
		\[
		z''(t) = \underbar w(\delta)(1-2\epsilon)\{\phi(z(t)) - \phi(z(t-\delta))\} > 0, \quad \mbox{ for each }t > T_1 + \delta,
		\]
		so $z$ is convex on $(T_1 + \delta,\infty)$. Now use the convexity of $z$ to show that
		\begin{align}\label{est.first_order_lower}
		\left(z'(t) \right)^2 &= \underbar{w}(\delta)(1-2\epsilon) \int_{t-\delta}^t \phi(z(s))z'(t)\,ds\geq \underbar{w}(\delta)(1-2\epsilon) \int_{t-\delta}^t \phi(z(s))z'(s)\,ds \nonumber\\
		&= \underbar w(\delta)(1-2\epsilon) \left\{ \bar{\Phi}(z(t)) - \bar{\Phi}(z(t-\delta)) \right\}, \quad \mbox{ for each }t > T_1 + 2\delta,
		\end{align}
		where $\bar{\Phi}(x) = \int_0^x \phi(s)\,ds$. The function $q$ given by $q(t) = z(t+T_1+\delta)$ for $t \geq -T_1-\delta$ solves \eqref{eq.bounded_delay} with $C = \underbar w(\delta)(1-2\epsilon)$ and $\psi = y/2$. Hence Lemmas \ref{lemma.bounded_delay_1} and \ref{lemma.bounded_delay_2} apply to $q$, and therefore
		\[
		\lim_{t\to\infty}\frac{q(t-\delta)}{q(t)} =0, \quad \lim_{t\to\infty}\frac{\bar{\Phi}(q(t-\delta))}{\bar{\Phi}(q(t))} =0.
		\]
		It follows that 
		$
		\lim_{t\to\infty}\bar{\Phi}(z(t-\delta))/\bar{\Phi}(z(t)) =0
		$
		and combining this limit with \eqref{est.first_order_lower} yields
		\[
		\liminf_{t\to\infty}\frac{(z'(t))^2}{\bar{\Phi}(z(t))} \geq \underbar w(\delta)(1-2\epsilon)>0.
		\]
		Thus there exists a $T^*(\epsilon)>0$ such that for each $\epsilon \in (0,1/2)$
		\[
		\frac{(z'(t))^2}{\bar{\Phi}(z(t))} > (1-\epsilon)\underbar w(\delta)(1-2\epsilon), \quad t \geq T^*(\epsilon).
		\]
		Taking the square root across the inequality above and integrating from $T^*(\epsilon)$ to some fixed $t > T^*(\epsilon)$ we obtain
		\begin{align*}
		\int_{T^*}^t \frac{z'(s)ds}{\bar{\Phi}(z(s))^{1/2}} = \int_{z(T^*)}^{z(t)}\bar{\Phi}(u)^{-1/2} \geq (t-T^*)\sqrt{(1-\epsilon) \underbar w(\delta)(1-2\epsilon)}, \quad t > T^*.
		\end{align*}
		Since $z(t) \to \infty $ as $t \to \infty$ and $\bar{F}(x) \sim \bar{\Phi}(x)$ as $x\to\infty$, taking the liminf in the inequality above gives 
		\[
		\int_{z(T^*)}^{\infty}\bar{F}(u)^{-1/2} = \infty,
		\]
		in contradiction to \eqref{condition.finite}. Therefore $T<\infty$, as claimed.
	\end{proof}
	\begin{lemma}\label{lemma.bounded_delay_no_explosion}
		Let $C>0$ and $\delta>0$, and suppose that \eqref{eq.f1_ch6} and \eqref{eq.psi} hold. If \eqref{condition.infinite} holds, then solutions to \eqref{eq.bounded_delay} obey $z \in C([-\delta,\infty);(0,\infty))$. Similarly, solutions to \eqref{eq.bounded_delay_kernel} obey $y \in C([-\delta,\infty);(0,\infty))$.
	\end{lemma}
	\begin{proof}[Proof of Lemma \ref{lemma.bounded_delay_no_explosion}]
		First consider equation \eqref{eq.bounded_delay}. By \eqref{eq.psi}, there exists a $T \in (0,\infty]$ such that $z \in C([-\delta,T);(0,\infty))$ and $\lim_{t\to T^-}z(t)=\infty$. Suppose $T \in (0,\infty)$. By \eqref{eq.f1_ch6}, there exists an increasing, continuous function $\phi$ such that $f(x) < \kappa \,\phi(x)$ for some $\kappa>0$, for each $x>0$. Define $\phi_\kappa(x) = \kappa\, \phi(x)$ for each $x>0$ and note that
		\[
		\int_1^\infty \frac{du}{\sqrt{\int_0^u \phi_\kappa(s)\,ds}} = \infty
		\]
		is equivalent to \eqref{condition.infinite}, since $f \sim \phi$.
		Let $\psi = 1 + \sup_{s \in [0,T/2]}z(s)$ and define the function $\alpha$ by
		\[
		\alpha'(t) = \sqrt{2K_1 \int_{1}^{\alpha(t)}\phi_\kappa(u)\,du }, \quad t \geq 0; \quad \alpha(t)=\psi, \quad t \leq 0,
		\]   
		with 
		$
		K_1 = \max\left\{2, \, (\delta\,\phi_\kappa(\psi))^2/2 \int_1^\psi \phi_\kappa(u)\,du \right\}.
		$ Both $\psi$ and $K_1$ are larger than $1$, and \eqref{condition.infinite} implies $\alpha \in C((-\infty,\infty);(0,\infty))$. In fact, due to the continuity of $\phi_\kappa$, $\alpha \in C^2((0,\infty);(0,\infty))$. Furthermore, $\alpha'(t)>0$ for $t \geq 0$ and due to our choice of $\psi$, $\alpha(t) > z(t)$ for each $t \in [-\delta,T/2]$. Now consider the function
		\[
		A_\alpha(t) := \int_{t-\delta}^t \phi_\kappa(\alpha(u))\,du, \quad t \geq 0.
		\]
		Differentiating $A_\alpha$, estimating, and using the fact that $\alpha''(t) = K_1 \phi_\kappa(\alpha(t))$ for $t > 0$ yields
		\[
		A_\alpha '(t) = \phi_\kappa(\alpha(t)) - \phi_\kappa(\alpha(t-\delta)) < K_1 \phi_\kappa(\alpha(t)) = \alpha''(t), \quad t > 0.
		\]
		Integrating from $0$ to $t$ we obtain
		\begin{align*}
		A_\alpha(t) - A_\alpha(0) &= \int_{t-\delta}^t \phi_\kappa(\alpha(u))\,du - \int_{-\delta}^0 \phi_\kappa(\alpha(u))\,du \\ &\leq \alpha'(t) - \sqrt{2K_1 \int_{1}^{\alpha(0)}\phi_\kappa(u)\,du } = \alpha'(t) - \alpha'(0), \quad t \geq 0.
		\end{align*}
		Rearrangement shows that the inequality above is equivalent to 
		\begin{equation}\label{est.inequality_weak}
		\alpha'(t) \geq \int_{t-\delta}^t \phi_\kappa(\alpha(u))\,du  + \sqrt{2K_1 \int_{1}^{\psi}\phi_\kappa(u)\,du } - \int_{-\delta}^0 \phi_\kappa(\psi)\,du, \quad t \geq 0.
		\end{equation}
		$
		\sqrt{2K_1 \int_{1}^{\psi}\phi_\kappa(u)\,du } - \int_{-\delta}^0 \phi_\kappa(\psi)\,du > 0
		$
		if and only if $K_1 > (\delta\,\phi_\kappa(\psi))^2 / 2 \int_1^\psi \phi_\kappa(u)\,du$, which is guaranteed by our earlier choice of $K_1$. Hence inequality \eqref{est.inequality_weak} implies that 
		\begin{equation}\label{est.key_breakdown}
		\alpha'(t) > \int_{t-\delta}^t \phi_\kappa(\alpha(u))\,du, \quad t \geq 0.
		\end{equation}
		Now suppose there is a minimal $T_B \in (T/2,T)$ such that $\alpha(T_B)= z(T_B)$. Since $\alpha(t)>z(t)$ for each $t \in [-\delta,T/2]$, it must be the case that $z'(T_B) \geq \alpha'(T_B)$. Thus
		\begin{align*}
		z'(T_B) = \int_{T_B-\delta}^{T_B} f(z(u))\,du \leq \int_{T_B-\delta}^{T_B} \phi_\kappa(\alpha(u))\,du < \alpha'(T_B),
		\end{align*}
		where the final \emph{strict} inequality follows from \eqref{est.key_breakdown}. But this implies that $z'(T_B) < \alpha'(T_B) \leq z'(T_B)$, a contradiction. Thus $z(t) < \alpha(t)$ for each $t \in [-\delta,\infty)$ and therefore $T=\infty$  since \eqref{condition.infinite} ensures that $\alpha$ is bounded on compact intervals.
		
		Now consider \eqref{eq.bounded_delay_kernel}. By hypothesis, $y \in C([-\delta,T);(0,\infty))$ for some $T>0$ and thus
		\begin{align*}
		y'(t)  = \int_0^\delta w(u)f(y(t-u))\,du < 2\sup_{s \in [0,\delta]}w(s) \int_0^\delta f(y(t-u))\,du, \quad t\in (0,T).
		\end{align*}
		Define $\bar{w}(\delta) = \sup_{s \in [0,\delta]}w(s) > 0$ and hence define the upper comparison solution $z$ by
		\[
		z'(t) = 2\bar{w}(\delta) \int_{t-\delta}^t \phi_\kappa(z(s))\,ds, \quad t \geq 0;\quad z(t) = \sup_{u \in [-\delta,0]}\psi(u) + 1, \quad t \in [-\delta,0].
		\]
		By the arguments above, $z \in C([-\delta,\infty);(0,\infty))$ and, by construction, $z(t) > y(t)$ for each $t \in [-\delta,T)$. Hence $y$ cannot explode in finite--time and the claim is proven.
	\end{proof}
	Our final lemma identifies the growth rate of solutions to \eqref{eq.bounded_delay}. The corresponding results for \eqref{eq.volterra_ch6} and \eqref{eq.volterra_forced} consist of carefully constructing comparison solutions using equations of the form of \eqref{eq.bounded_delay} and then invoking this lemma.
	\begin{lemma}\label{lemma.bounded_rate}
		Suppose that the hypotheses of Lemma \ref{lemma.bounded_delay_1} hold. If $f$ preserves superexponential growth, then the solution $z \in C([-\delta,\infty);(0,\infty))$ to \eqref{eq.bounded_delay} obeys
		\[
		\lim_{t\to\infty}\frac{F_U(z(t))}{t} = \sqrt{2C}.
		\] 
	\end{lemma}
	\begin{proof}[Proof of Lemma \ref{lemma.bounded_rate}]
		Due to the continuity of $f$, $z \in C^2((\delta,\infty);(0,\infty))$ and 
		\begin{equation}\label{eq.z_diff_superexp}
		z''(t) = C f(z(t)) - Cf(z(t-\delta)), \quad t > \delta.
		\end{equation}
		By Lemma \ref{lemma.bounded_delay_1}, $\lim_{t\to\infty}z(t-\delta)/z(t) = 0$ and hence, because $f$ preserves superexponential growth,
$
		\lim_{t\to\infty}f(z(t-\delta))/f(z(t)) = 0.
$ Hence it follows from \eqref{eq.z_diff_superexp} that 
		\begin{align}\label{z_prime_limit}
		\lim_{t\to\infty}\frac{z''(t)}{C f(z(t))} = 1.
		\end{align}
		It follows from \eqref{eq.f1_ch6} that $\int_0^{z(t)}f(u)\,du \to \infty$ as $t\to\infty$ and hence $z'(t) \to \infty$ as $t\to\infty$ by integration of \eqref{z_prime_limit}. Now use L'H\^{o}pital's rule to show that
		\[
		\lim_{t\to\infty}\frac{(z'(t))^2}{2C \int_0^{z(t)}f(u)\,du} = \lim_{t\to\infty}\frac{ z''(t)}{C f(z(t))} = 1.
		\]
		Therefore 
		\[
		\lim_{t\to\infty}\frac{z'(t)}{\sqrt{\int_0^{z(t)}f(u)\,du }} = \sqrt{2C}.
		\]
		It follows that for each $\epsilon>0$ there exists $T^*(\epsilon)>0$ such that 
		\[
		\sqrt{2C} - \epsilon < \frac{z'(t)}{\sqrt{\int_0^{z(t)}f(u)\,du }} < \epsilon + \sqrt{2C}, \quad t \geq T^*(\epsilon).
		\]
		Suppose $t > T^*(\epsilon)$ and integrate the inequality above to yield
		\[
		(\sqrt{2C} - \epsilon)(t-T^*) < \int_{T^*}^t \frac{z'(u)\,du}{\sqrt{\int_0^{z(u)}f(s)\,ds }} < (\epsilon + \sqrt{2C})(t-T^*), \quad t > T^*(\epsilon).
		\]
		By making the substitution $y = z(u)$ it is straightforward to show that 
		\[
		(\sqrt{2C} - \epsilon)\frac{t-T^*}{t} + \frac{F_U(z(T^*))}{t} < \frac{F_U(z(t))}{t} < (\epsilon + \sqrt{2C})\frac{t-T^*}{t} + \frac{F_U(z(T^*))}{t},\quad t > T^*(\epsilon),
		\]
		Let $t \to \infty$ and then $\epsilon \to 0^+$ in the inequalities above to complete the proof.
	\end{proof}
	\section{Proofs of Main Results}\label{proofs_main}
	\begin{proof}[Proof of Theorem \ref{thm.sufficient_blow_up}]\textbf{Sufficiency:} Suppose \eqref{condition.finite} holds. By the usual considerations, $x \in C([0,T);(0,\infty))$ for some $T \in (0,\infty]$. Suppose $T=\infty$ and let $\tau>0$ be arbitrary. By \eqref{eq.w} and positivity,
		\begin{align*}
		x'(t) &= \int_{0}^t w(t-s)f(x(s))\,ds > \int_{t-\tau}^t w(t-s)f(x(s))\,ds, \quad t \geq \tau.
		\end{align*}
		Let $\phi$ denote any monotone increasing, continuous function obeying $f(x) \sim \phi(x)$ as $x\to\infty$. Since $x(t)\to\infty$ as $t \to \infty$, for each $\epsilon\in (0,1)$ there exists $T_1(\epsilon)>0$ such that $f(x(t)) > (1-\epsilon)\phi(x(t))$ for each $t \geq T_1(\epsilon)$. Hence
		\begin{align*}
		x'(t) > (1-\epsilon)\int_{t-\tau}^t w(t-s)\phi(x(s))\,ds, \quad t \geq T_1(\epsilon) + \tau.
		\end{align*}
		Define the lower comparison solution $y$ by
		\[
		y'(t) = (1-\epsilon)\int_{t-\tau}^t w(t-s)\phi(y(s))\,ds, \quad t \geq T_1(\epsilon) + \tau; \quad y(t) = x_L(t), \quad t \in [0,T_1+\tau],
		\] 
		where $x_L$ obeys $x_L'(t) = \int_0^t w(t-s)f(x_L(s))\,ds$ for $t \in [0,T_1+\tau]$ and $x_L(0) = x(0)/2$. By construction, $y(t) < x(t)$ for $t \geq 0$. Let $y_\tau(t) = y(t+\tau+T_1)$ for each $t \geq - T_1 - \tau$ and note that $y_\tau$ solves \eqref{eq.bounded_delay_kernel} with $\delta = \tau+T_1$ and $\psi=x_L$. Hence Lemma \ref{lemma.bounded_delay_kernel_explodes} applies to $y_\tau$ and there exists a $T_\tau < \infty$ such that $\lim_{t\to T_\tau^-}y_\tau(t) = \infty$, contradicting the assumption that $T=\infty$ and completing the proof.
		
		\textbf{Necessity:} Suppose \eqref{condition.infinite} holds. As usual, our hypotheses guarantee a well defined solution to \eqref{eq.volterra_ch6} on some maximal interval $[0,T)$ with $T \in (0,\infty]$. Assume, contrary to our claim, that $T<\infty$. Let $\delta \in (0,T)$ and estimate the derivative of $x$ for $t \in (\delta,T)$ as follows:
		\begin{align}\label{est.no_expl_initial}
		x'(t) \leq \bar{w}(\delta)\int_{t-\delta}^t f(x(s))\,ds + \bar{M}(\delta),
		\end{align}
		where $\bar{w}(\delta)=\sup_{s \in [0,\delta]}w(s)$ and $\bar{M}(\delta) = \sup_{t \in [0,T]}\int_0^{t-\delta} w(t-s)f(x(s))\,ds$. Note that
		\[
		\limsup_{t\to T^-}\frac{\bar{M}(\delta)}{\bar{w}(\delta)\int_{t-\delta}^t f(x(s))\,ds} =: C(\delta) \in  [0,\infty).
		\] 
	Combining the limit superior above with \eqref{est.no_expl_initial} yields
		\[
		\limsup_{t\to T^-}\frac{x'(t)}{\bar{w}(\delta)\int_{t-\delta}^t f(x(s))\,ds} \leq 1 + C(\delta)<\infty, \quad \mbox{ for each }\delta \in (0,T).
		\] 
		Thus, for each $\epsilon>0$, there exists $T^*(\epsilon)\in (\delta,T)$ such that
		\[
		x'(t) < (1+\epsilon)(1 + C(\delta))\bar{w}(\delta)\int_{t-\delta}^t f(x(s))\,ds, \quad t \in [T^*(\epsilon),T).
		\]
		Taking $\epsilon = \delta$ in the estimate above gives
		\[
		x'(t) < (1+\delta)(1 + C(\delta))\bar{w}(\delta)\int_{t-\delta}^t f(x(s))\,ds, \quad t \in [T^*(\delta),T).
		\]
		By hypothesis, there is an increasing, continuous function $\phi$ such that $f(x)<\kappa\, \phi(x)$ for some $\kappa>0$, for each $x>0$. As before let $\phi_\kappa(x) = \phi_\kappa(x)$ for each $x>0$. Hence
		\[
		x'(t) < (1+\delta)(1 + C(\delta))\bar{w}(\delta)\int_{t-\delta}^t \phi_\kappa(x(s))\,ds, \quad t \in [T^*(\delta),T).
		\]
		Now define the upper comparison solution $z$ according to 
		\begin{align*}
		z'(t) &= (1+2\delta)(1 + C(\delta))\bar{w}(\delta)\int_{t-\delta}^t \phi_\kappa(z(s))\,ds, \quad t \geq 0, 
		\end{align*}
		with $ z(t) = Z^* := 1+ \sup_{u \in [0,T^*(\delta)]}x(u)$ for $ t \in [-\delta,0]$. By construction, $x(t) < z(t)$ for all $t \in [0,T)$. However, since $z$ solves \eqref{eq.bounded_delay} with $C = (1+2\delta)(1 + C(\delta))\bar{w}(\delta)$ and $\psi \equiv Z^*$, Lemma \ref{lemma.bounded_delay_no_explosion} implies that $z \in C([-\delta,\infty);(0,\infty))$. Therefore the assumption that $T<\infty$ leads to a contradiction and the proof is complete. 
	\end{proof}
	\begin{proof}[Proof of Theorem \ref{thm.expl_rate}]
		By hypothesis there exists $T\in (0,\infty)$ such that $x \in C([0,T);(0,\infty))$ and $\lim_{t \to T^-}x(t)=\infty$. First show that $\lim_{t \to T^-}x'(t) = \infty$. For an arbitrary $\delta \in (0,T)$, construct a lower bound on $x'$ of the form
		\begin{equation}\label{eq.lower_bound_derivative}
		x'(t) > \kappa \,\underline{w}(\delta)\int_{t-\delta}^t \phi(x(u))\,du + \underbar{M}(\delta), \quad t \in (\delta,T),
		\end{equation}
		where $\underbar{M}(\delta)=\inf_{t\in [0,T]}\int_0^{t-\delta}w(t-s)f(x(s))\,ds$, $\kappa>0$, and $\phi$ is monotone increasing and continuous. By hypothesis, $\lim_{t\to T^-}\phi(x(t))=\infty$ and hence
		\[
		\lim_{t\to T^-} \frac{d}{dt}\int_{t-\delta}^t \phi(x(u))\,du = \infty.
		\]
		Thus the function $t \mapsto\int_{t-\delta}^t \phi(x(u))\,du$ is increasing on some interval $(T^*,T)$ and must have a limit as $t\to T^-$. However, if $\lim_{t\to T^-}\int_{t-\delta}^t \phi(x(u))\,du$ is finite, integration of \eqref{est.no_expl_initial} yields
		$
		x(t) \leq A + Bt \mbox{ for }t \in (T^*_1,T),
		$
		for some positive constants $A$ and $B$. But the solution blows-up in finite--time, a contradiction. Therefore, $\lim_{t\to T^-}\int_{t-\delta}^t \phi(x(u))\,du=\infty$ and hence $\lim_{t \to T^-}x'(t) = \infty$, as claimed.
		
		Now let $\delta \in (0,T)$ be arbitrary and estimate as follows:
		\begin{align}\label{est.upper.initial.rate}
		x'(t) \leq \bar{w}(\delta)\int_{t-\delta}^t f(x(u))\, du + \bar{M}(\delta) \quad t \in (\delta,T),
		\end{align}
		where $\bar{w}(\delta) = \sup_{u \in [0,\delta]}w(u)$ and $\bar{M}(\delta) = \sup_{t \in [0,T]}\int_0^{t-\delta} w(t-s)f(x(s))\,ds$. Note that $\bar{M}(\delta)$ is finite for each $\delta \in (0,T)$. By \eqref{est.upper.initial.rate} and the fact that $x'(t)\to\infty$ as $t \to T^-$, we have $\int_{t-\delta}^t f(x(u))\,du \to \infty$ as $t\to T^-$, for each $\delta \in (0,T)$. Thus $\int_0^t f(x(u))\,du \to \infty$ as $t\to T^-$ and, by applying L'H\^{o}pital's rule,
		\[
		\lim_{t\to T^-}\frac{\int_{t-\delta}^t f(x(u))\,du}{\int_0^t f(x(u))\,du} = \lim_{t\to T^-}\frac{f(x(t)) -f(x(t-\delta)) }{f(x(t))} = 1, \quad \mbox{ for each }\delta \in (0,T).
		\] 
		Dividing across by $\bar{w}(\delta) \int_0^t f(x(u))\,du$  in \eqref{est.upper.initial.rate} and taking the limsup thus yields
		\[
		\limsup_{t\to T^-}\frac{x'(t)}{\bar{w}(\delta) \int_0^t f(x(u))\,du} \leq 1, \quad \mbox{ for each }\delta \in (0,T).
		\]
		Letting $\delta \to 0^+$ in the limit above shows that 
		\[
		\limsup_{t\to T^-}\frac{x'(t)}{w(0) \int_0^t f(x(u))\,du} \leq 1.
		\]
		Similarly, we can obtain the following lower estimate on the derivative
		\begin{align*}
		x'(t) > \int_{t-\delta}^t w(t-s)f(x(s))\, ds \geq \underbar w(\delta) \int_{t-\delta}^t f(x(u))\, du, \quad t \in (\delta,T),
		\end{align*}
		where $\underbar w(\delta) = \inf_{u \in [0,\delta]}w(u) > 0$. Following the same steps as above quickly reveals that 
		$
		\liminf_{t\to T^-}x'(t)/w(0) \int_0^t f(x(u))\,du \geq 1.
		$
		Therefore
		\begin{align}\label{limit.first_ode}
		\lim_{t\to T^-}\frac{x'(t)}{w(0)\int_0^t f(x(s))\,ds} = 1.
		\end{align}
		We claim that \eqref{limit.first_ode} implies
		\begin{align}\label{limit.second_ode}
		\lim_{t\to T^-}\frac{x'(t)}{\sqrt{2w(0)\int_0^{x(t)} f(s)\,ds}} = 1.
		\end{align}
		Using \eqref{limit.first_ode}, \eqref{limit.second_ode} is equivalent to 
		\[
		\lim_{t\to T^-}\frac{w(0)\int_0^t f(x(s))\,ds}{\sqrt{2w(0)\int_0^{x(t)} f(s)\,ds}} = 1.
		\] 
		Letting $I(t) = \int_0^t f(x(s))\,ds$, the limit above is in turn equivalent to
		\[
		\lim_{t\to T^-}\frac{[w(0)I(t)]^2}{2w(0)\int_0^{x(t)} f(s)\,ds} = 1.
		\] 
		However, since $I(t) \to \infty$ as $t\to T^-$ and $\int_0^{x(t)} f(s)\,ds \to \infty$ as $t\to T^-$, applying L'H\^{o}pital's rule yields
		\[
		\lim_{t\to T^-}\frac{[(w(0)I(t)]^2}{2w(0)\int_0^{x(t)} f(s)\,ds} = \lim_{t\to T^-}\frac{2[w(0)]^2 I(t)I'(t)}{2w(0)x'(t)f(x(t))} = \lim_{t\to T^-}\frac{w(0)\int_0^t f(x(s))\,ds }{x'(t)} = 1,
		\]
		where the final equality follows from \eqref{limit.first_ode}. Thus \eqref{limit.first_ode} implies \eqref{limit.second_ode}, as claimed.
		
		By \eqref{limit.second_ode}, for each $\epsilon \in (0,1)$, there exists $\bar{T}(\epsilon) \in (0,T)$ such that 
		\[
		1-\epsilon < \frac{x'(t)}{\sqrt{2w(0)\int_0^{x(t)} f(s)\,ds}} < 1+\epsilon, \quad t \in (\bar{T},T).
		\]
		Let $t$ and $T_L$ be such that $\bar{T} < t < T_L < T$ and integrate the inequalities above from $t$ to $T_L$; this yields
		\[
		(1-\epsilon)(T_L-t)\sqrt{2w(0)} < \int_{t}^{T_L}\frac{x'(u)\,du}{\sqrt{\int_0^{x(u)} f(s)\,ds}} < (1+\epsilon)(T_L-t)\sqrt{2w(0)},
		\]
		for $\bar{T} < t < T_L < T$. Make the substitution $y = x(u)$ in the integral to obtain
		\[
		(1-\epsilon)(T_L-t)\sqrt{2w(0)} < \int_{x(t)}^{x(T_L)}\frac{dy}{\sqrt{\int_0^y f(s)\,ds}} < (1+\epsilon)(T_L-t)\sqrt{2w(0)},
		\]
		for $\bar{T} < t < T_L < T$. Now let $T_L \to T^-$ and divide across by $T-t$ to show that
		\[
		(1-\epsilon)\sqrt{2w(0)} < \frac{1}{T-t}\int_{x(t)}^{\infty}\frac{dy}{\sqrt{\int_0^y f(s)\,ds}}  = \frac{F_B(x(t))}{T-t} < (1+\epsilon)\sqrt{2w(0)}, \quad \bar{T} < t < T.
		\]
		Let $\epsilon \to 0^+$ in the inequalities above to complete the proof.
	\end{proof}
	\begin{proof}[Proof of Theorem \ref{thm.non_expl_rate}]
		By hypothesis, $x \in C([0,\infty);(0,\infty))$. Let $\delta > 0$ be arbitrary and estimate as follows:
		\begin{align*}
		x'(t) &> \int_{t-\delta}^t w(t-s)f(x(s))\,ds = \int_{0}^\delta w(u)f(x(t-u))\,du 
		\geq \underbar w(\delta)\int_{t-\delta}^t f(x(s))\,ds, \quad t \geq \delta,
		\end{align*}
		where $\underbar w(\delta) = \inf_{u \in [0,\delta]}w(u) > 0$. By \eqref{eq.f1_ch6}, there exists a continuous, increasing function $\phi$ such that, for each $\epsilon\in (0,1)$, $f(x(s)) > (1-\epsilon)\phi(x(s))$ for all $s \geq T(\epsilon)$. Hence
		\[
		x'(t) > (1-\epsilon)\underbar w(\delta)\int_{t-\delta}^t \phi(x(s))\,ds, \quad t \geq T(\epsilon) + \delta.
		\]
		Now define the lower comparison solution $z_-$ by
		\[
		z_- '(t) = (1-\epsilon)\underbar w(\delta)\int_{t-\delta}^t \phi(z_-(s))\,ds, \quad t \geq T + \delta; \quad z_-(t) = x(t)/2, \quad t \in [0,T+\delta].
		\]
		It can be shown that $z_-(t) < x(t)$ for all $t \geq 0$ and applying Lemma \ref{lemma.bounded_rate} to $z_-$ shows that
		\begin{align}\label{limit.lower.rate}
		\sqrt{2(1-\epsilon) \underbar w(\delta)} = \liminf_{t\to\infty}\frac{F_U(z_-(t))}{t} \leq \liminf_{t\to\infty}\frac{F_U(x(t))}{t},
		\end{align}
		for each $\delta>0$. Let $\delta \to 0^+$ and then $\epsilon \to 0^+$ in the inequality above to obtain
		\[
		\liminf_{t\to\infty}\frac{F_U(x(t))}{t} \geq \sqrt{2 w(0)}.
		\]
		We now tackle the corresponding limsup. By \eqref{eq.f1_ch6}, there exists a continuous, increasing function $\phi$ such that, for each $\epsilon>0$, $f(x(t))< (1+\epsilon)\phi(x(t))$ for each $t \geq T_1(\epsilon)$. Furthermore, we can find a $\kappa>0$ such that $f(x)<\kappa\,\phi(x)$ for each $x>0$. For $\delta > 0$,
		\begin{align*}
		x'(t) &\leq \kappa\,\phi(x(t-\delta)) \int_{0}^{t-\delta}w(t-s)\,ds + \int_{t-\delta}^t w(t-s)f(x(s))\,ds, \quad t \geq \delta+T_1(\epsilon),
		\end{align*}
		where $\bar{w}(\delta) = \sup_{u \in [0,\delta]}w(u) > 0$. Thus
		\begin{align}\label{est.upper_rate_1}
		x'(t) < \kappa\,{\mathcal W} \phi(x(t-\delta)) + (1+\epsilon)\bar{w}(\delta) \int_{t-\delta}^t \phi(x(s))\,ds, \quad t \geq T_1(\epsilon) + \delta.
		\end{align}
		Now make the following lower estimate on the second term in \eqref{est.upper_rate_1}:
		\[
		\bar{w}(\delta)\int_{t-\delta}^t \phi(x(s))\,ds > \bar{w}(\delta)\int_{t-\delta/2}^t \phi(x(s))\,ds \geq \frac{\delta}{2} \phi(x(t-\delta/2)), \quad t \geq T_1(\epsilon) + \delta. 
		\]   
		Hence 
		\[
		0 \leq \limsup_{t\to\infty}\frac{\kappa\,{\mathcal W}\, \phi(x(t-\delta))} {\bar{w}(\delta)\int_{t-\delta}^t \phi(x(s))\,ds} \leq \lim_{t\to\infty}\frac{2 {\kappa\,\mathcal W}}{\delta} \frac{\phi(x(t-\delta))}{\phi(x(t-\delta/2))} = 0, 
		\]
		where the final limit is established by repeating verbatim the argument from Lemma \ref{lemma.bounded_delay_1}. Combining the limit above with \eqref{est.upper_rate_1} then yields
		\[
		\limsup_{t\to\infty}\frac{x'(t)}{\bar{w}(\delta) \int_{t-\delta}^t \phi(x(s))\,ds} \leq 1+\epsilon.
		\]
		Hence, for each $\epsilon>0$, there exists $T^*(\epsilon,\delta)>0$ such that 
		\[
		x'(t) < (1+\epsilon)^2\bar{w}(\delta) \int_{t-\delta}^t \phi(x(s))\,ds, \quad t \geq T^*(\epsilon,\delta).
		\] 
		Fix $\epsilon=\delta>0$ so that 
		$
		x'(t) < (1+\delta)^2\bar{w}(\delta) \int_{t-\delta}^t \phi(x(s))\,ds\mbox{ for } t \geq T^*(\delta).
		$ 
		Define $z_+$ by
		\begin{align*}
		z_+'(t) &= (1+\delta)^2\bar{w}(\delta) \int_{t-\delta}^t \phi(z_+(s))\,ds, \quad t \geq 0; \quad z_+(t)= \sup_{u \in [0,T^*(\delta)]}x(u)+1, \quad t \in [-\delta,0].
		\end{align*}
		By construction, $z_+(t) > x(t)$ for each $t \geq 0$ and, by applying Lemma \ref{lemma.bounded_rate}, we obtain
		\[
		\limsup_{t\to\infty}\frac{F_U(x(t))}{t} \leq \limsup_{t\to\infty}\frac{F_U(z_+(t))}{t} \leq \sqrt{2(1+\delta)^2\bar{w}(\delta)}, \mbox{ for each }\delta>0,
		\]
		once more using that $\phi\sim f$. Therefore, since $\delta > 0$ was arbitrary,
		\[
		\limsup_{t\to\infty}\frac{F_U(x(t))}{t} \leq \sqrt{2 w(0)},
		\]
		and combining the inequality above with the corresponding liminf yields the result.
	\end{proof}
	\begin{proof}[Proof of Theorem \ref{thm.expl_rate_perturbed}]
		We first show that \eqref{condition.finite} is a sufficient condition for the finite--time blow--up of solutions to \eqref{eq.volterra_forced}. Suppose $T=\infty$. For an arbitrary $\tau>0$, we have the trivial lower estimate
		$
		x'(t) > \int_{t-\tau}^t w(t-s)f(x(s))\,ds + h(t)$ for $t \geq \tau.
		$ 
		By \eqref{eq.f1_ch6} there exists an increasing, positive function $\phi$ asymptotic to $f$ and a finite, positive constant $C$ such that
		$
		C = \inf_{x \in \left[x(0)/2,\infty\right)}f(x)/\phi(x)
	$.
		Hence $C\,\phi(x) \leq f(x)$ for $x \geq x(0)/2$. Now let $\varphi(x) = C\,\phi(x)$ for each $x \geq x(0)/2$. Thus
		\[
		x'(t) > \int_{t-\tau}^t w(t-s)\varphi(x(s))\,ds + h(t),\quad t \geq \tau,
		\]
		because $x(t)\geq x(0)>x(0)/2$ for all $t\geq 0$. Integration then yields
		\begin{align}\label{est.integral_pert}
		x(t) \geq x(\tau) + \int_{\tau}^t\int_{u-\tau}^u w(u-s)\varphi(x(s))\,ds\,du + H_\tau(t), \quad t \geq \tau,
		\end{align}
		where $H_\tau(t) = \int_\tau^t h(s)\,ds$. Define the lower comparison solution $y$ by
		\[
		y(t) = y(\tau) + \int_\tau^t \int_{u-\tau}^u w(u-s)\varphi(y(s))\,ds\,du, \quad t \geq \tau;\quad y(t) = x(0)/2, \quad t \in [0,\tau].
		\]
		Of course, $y$ also obeys the delayed integro--differential equation
		\[
		y'(t) = \int_{t-\tau}^t w(t-s)\varphi(y(s))\,ds,\quad t \geq \tau;\quad y(t) = x(0)/2, \quad t \in [0,\tau].
		\] 
		Note that $x(t) \geq x(0)$ for each $t>0$ due to \eqref{eq.H}, and $y(t)< x(t)$ for each $t \in [0,\tau]$ by construction. Use the continuity of $h$ to choose $\tau>0$ sufficiently small that $\int_0^\tau h(s)\,ds \leq x(0)/4$ and suppose $T_B>\tau$ is the minimal time such that $y(T_B)= x(T_B)$. Thus
		\begin{align*}
		y(T_B) &= x(T_B) \geq x(\tau) + H_\tau(T_B) + \int_{\tau}^{T_B}\int_{u-\tau}^u w(u-s)\varphi(x(s))\,ds\,du  \\
		&\geq x(0) + H(T_B) - \int_0^\tau h(s)\,ds + \int_{\tau}^{T_B}\int_{u-\tau}^u w(u-s)\varphi(x(s))\,ds\,du \\
		&\geq x(0) - \int_0^\tau h(s)\,ds + \int_{\tau}^{T_B}\int_{u-\tau}^u w(u-s)\varphi(y(s))\,ds\,du \\
		&> x(0)/2 + \int_{\tau}^{T_B}\int_{u-\tau}^u w(u-s)\varphi(y(s))\,ds\,du = y(T_B),
		\end{align*}
		a contradiction. Therefore $y(t)<x(t)$ for each $t \geq 0$. The proof of necessity in Theorem \ref{thm.sufficient_blow_up} now shows that $T=\infty$ produces a contradiction and hence that $T \in (0,\infty)$, as required.
		
		We have shown $x \in C([0,T);(0,\infty))$ for some $T \in (0,\infty)$ with $\lim_{t \to T^-}x(t)=\infty$, so we now proceed to show that \eqref{eq.blow_up_rate_preserved} holds. Since $h$ is continuous, there exists $a_1>0$ such that
		$
		|h(t)| \leq a_1 \mbox{ for all } t \in [0,T].
		$
		Following the line of argument from the proof of Theorem \ref{thm.expl_rate} yields the upper estimate
		\[
		x'(t) \leq a_1 + \bar{w}(\delta)\int_{t-\delta}^t f(x(s))\,ds + \int_0^{t-\delta} w(t-s)f(x(s))\,ds, \quad t \in (\delta,T),\quad \mbox{for each }\delta \in (0,T),
		\]
		with $\bar{w}(\delta) = \sup_{u \in [0,\delta]}w(u)$. Now let 
		$
		a_2(\delta) = \sup_{t \in [\delta,T]}\int_0^{t-\delta} w(t-s)f(x(s))\,ds
		$
		and note that $a_2(\delta)$ is bounded for each $\delta \in (0,T)$. Thus
		\begin{align}\label{est.upper_pert}
		x'(t) < a_3(\delta) +  \bar{w}(\delta)\int_{t-\delta}^t f(x(s))\,ds, \quad t \in (\delta,T),
		\end{align}
		where $a_3(\delta) = 1 + a_1 + a_2(\delta)$. Reuse the arguments from the proof of Theorem \ref{thm.expl_rate} to obtain
		\begin{align}\label{est.lower_pert}
		x'(t) > -a_1 + \underbar w(\delta) \int_{t-\delta}^t f(x(s))\, ds, \quad t \in (\delta,T),
		\end{align}
		where $\underbar w(\delta) = \inf_{u \in [0,\delta]}w(u)$. Define $I(t) = \int_0^t f(x(s))\,ds$ for $t \in [0,T)$. $I$ is increasing due to the positivity of $f \circ x$ and hence $\lim_{t\to T^-}I(t)$ exists. Suppose $\lim_{t\to T^-}I(t) = I^* \in (0,\infty)$. By positivity, $\int_{t-\delta}^t f(x(s))\,ds < I(t) \leq I^*$ for $t \in (\delta,T)$ and thus
		\[
		-a_1 < x'(t) < a_3(\delta) + \bar{w}(\delta)I^*, \quad t \in (\delta,T).
		\] 
		Integration of the inequalities above rules out the finite--time explosion of $x$, a contradiction. Therefore $\lim_{t\to T^-}I(t) = \infty$ and it follows from L'H\^{o}pital's rule that
		\[
		\lim_{t\to T^-}\frac{\int_{t-\delta}^t f(x(s))\,ds}{\int_0^t f(x(s))\,ds} = 1.
		\]
		Combining the limit above with \eqref{est.upper_pert} and \eqref{est.lower_pert} quickly establishes that
		\[
		1 \leq \liminf_{t\to\infty}\frac{x'(t)}{\underbar w(\delta)\int_0^t f(x(s))\,ds}, \quad \limsup_{t\to T^-}\frac{x'(t)}{\bar{w}(\delta)\int_0^t f(x(s))\,ds} \leq 1.
		\]
		Next let $\delta \to 0^+$ in the inequalities above to show that
		\[
		\lim_{t\to\infty}\frac{x'(t)}{w(0)\int_0^t f(x(s))\,ds} = 1.
		\]
		We are now in the same position as at equation \eqref{limit.first_ode} in the proof of Theorem \ref{thm.expl_rate}. Repeat verbatim the arguments which follow equation \eqref{limit.first_ode} to complete the proof.
	\end{proof}
	\noindent We state without proof several elementary lemmas required for the proof of Theorem \ref{thm.non_expl_rate_perturbed}.
	\begin{lemma}\label{lemma.ode_auxiliary}
		Suppose \eqref{eq.f2_ch6} and \eqref{condition.infinite} hold. Each solution, $y$, to the initial value problem
		\begin{equation}\label{eq.aux_ivp}
		y'(t) = \sqrt{\bar{F}(y(t))}, \quad t \geq 0; \quad y(0) = 1,
		\end{equation}
		obeys
		$
		\lim_{t\to\infty}y((1-\eta)t)/y(t) = 0 \mbox{ for each }\eta \in (0,1).
		$
	\end{lemma}
	\begin{lemma}\label{lemma.f_preserves_zero}
		Suppose $a$ and $b$ are continuous functions from $\mathbb{R}^+$ to $\mathbb{R}^+/\{0\}$ which obey $a(t)/b(t) \to 0$ as $t \to \infty$. If $f:\mathbb{R}^+ \mapsto (0,\infty)$ is a continuous function such that $x \mapsto f(x)/x$ is increasing for $x > 0$, then
		$
		\lim_{t\to\infty}f(a(t))/f(b(t)) = 0.
		$	
	\end{lemma}
	\begin{proof}[Proof of Theorem \ref{thm.non_expl_rate_perturbed}]
		Firstly, we claim that \eqref{condition.infinite} implies $x \in C([0,\infty);(0,\infty))$ in the presence of perturbations obeying \eqref{eq.H}. Under our standing hypotheses, there exists a $T \in (0,\infty]$ such that $x \in C([0,T);(0,\infty))$. Suppose $T < \infty$ and estimate $x'$ as follows:
		\begin{align*}
		x'(t) < \bar{w}(\delta)\int_{t-\delta}^t f(x(s))\,ds + \bar{M}(\delta) + h(t),\quad t \in (\delta,T), \quad \delta \in (0,T), 
		\end{align*}
		where $\bar{w}(\delta)=\sup_{s \in [0,\delta]}w(s)$ and $\bar{M}(\delta) = \sup_{t \in [0,T]}\int_0^{t-\delta} w(t-s)f(x(s))\,ds$. Since $h$ is continuous there exists $a_1>0$ such that $|h(t)| \leq a_1$ for all $t \in [0,T]$ and hence
		\begin{align}\label{est.no_expl_improved_pert}
		x'(t) < \bar{w}(\delta)\int_{t-\delta}^t f(x(s))\,ds + \bar{M}(\delta) + a_1,\quad t \in (\delta,T).
		\end{align}
		Beginning at equation \eqref{est.no_expl_initial}, repeat the argument from the proof of necessity in Theorem \ref{thm.sufficient_blow_up} verbatim to conclude that $x \in C([0,\infty);(0,\infty))$. Furthermore, by \eqref{eq.H} and a straightforward comparison argument with an equation of the form \eqref{eq.bounded_delay}, $\lim_{t\to \infty}x(t)=\infty$.
		
		We now show that $(i.)$ implies $(ii.)$. Suppose that \[\limsup_{t\to\infty}\frac{F_U(H(t))}{t} = K \in \left(0,\sqrt{2w(0)}\right).\] For each $\delta>0$, there exists $T_1(\delta)>0$ such that 
		$
		H(t) < F_U^{-1}\left((1+\delta) K t \right) = \bar{H}(t) \mbox{ for } t \geq T_1(\delta).
		$
			In the case $K=0$, we have
			$
			H(t) < F_U^{-1}\left(\delta t \right) =: \bar{H}(t)$ for $t \geq T_1(\delta)
			$  
			(and for each $\delta > 0$) and the proof proceeds analogously.
In integral form \eqref{eq.volterra_forced} reads
		$
		x(t) = x(0) + H(t) + \int_0^t W(t-s) f(x(s))\,ds$ for $t \geq 0,
		$
		where $W(t) = \int_0^t w(u)\,du$. Hence
		\begin{align}\label{upper_est_pert_1}
		x(t) &\leq x(0) + \bar{H}(t) + \int_{T_1}^t W(t-s) f(x(s))\,ds 
		+ {\mathcal W}\int_0^{T_1} f(x(s))\,ds, 
		\quad t > T_1(\delta).
		\end{align}
		Let $x^* = x(0) + {\mathcal W}\int_0^{T_1} f(x(s))\,ds$ and define the upper comparison solution $x_+$ by
		\[
		x_+(t) = 1 + x^* + \bar{H}(t) + \int_{T_1}^t W(t-s) f(x_+(s))\,ds, \quad t \geq T_1(\delta).
		\]
		By construction, $x(t) < x_+(t)$ for $t \geq T_1(\delta)$. Under \eqref{eq.f2_ch6}, $\bar{H} \in C^1((T_1,\infty);(0,\infty))$ and thus $x_+ \in C^1((T_1,\infty);(0,\infty))$. Hence differentiation yields
		\[
		x_+ '(t)  = \bar{H}'(t) + \int_{T_1}^t w(t-s) f(x_+(s))\,ds, \quad t > T_1(\delta).
		\]
		Since $x_+$ is increasing,
		$
		x_+ '(t) > \underbar w(\delta)\int_{t-\delta}^t f(x_+(s))\,ds$ for $t \geq T_1(\delta) + \delta,
		$
		where $\underbar w(\delta) = \inf_{u \in [0,\delta]}w(u)$. The argument from Lemma \ref{lemma.bounded_delay_1} shows that
		$
		\lim_{t\to\infty}x_+(t-\delta)/x_+(t) = 0$ for each $\delta>0.
		$
		Proposition \ref{LEMMA.PRESERVES} applies to $f$ since $x \mapsto f(x)/x$ is increasing and thus
	$
		\lim_{t\to\infty}f(x_+(t-\delta))/f(x_+(t)) = 0\mbox{ for each }\delta>0.
	$ Now estimate as follows:
		\begin{align}\label{est.x_+_upper}
		x_+'(t) &\leq \bar{H}'(t) + {\mathcal W} f(x_+(t-\delta)) + \bar{w}(\delta) \int_{t-\delta}^t f(x_+(s))\,ds,\quad t \geq T_1(\delta)+\delta,
		\end{align}
		where $\bar{w}(\delta) = \sup_{u \in [0,\delta]}w(u)$. As noted in earlier arguments
		\[
		\frac{f(x_+(t-\delta))}{\int_{t-\delta}^t f(x_+(s))\,ds} < \frac{f(x_+(t-\delta))}{\int_{t-\delta/2}^t f(x_+(s))\,ds} \leq \frac{2f(x_+(t-\delta))}{\delta f(x_+(t-\delta/2))},
		\]	
		and hence $\lim_{t\to\infty}f(x_+(t-\delta))/\int_{t-\delta}^t f(x_+(s))\,ds = 0$ for each $\delta>0$. Thus, from \eqref{est.x_+_upper},
		\[
		x_+'(t) < \bar{H}'(t) + \bar{w}(\delta)(1+\delta) \int_{t-\delta}^t f(x_+(s))\,ds, \quad t \geq T_2(\delta) > T_1(\delta)+\delta,
		\]
		where $T_2(\delta)$ is chosen sufficiently large that
		$
		{\mathcal W}f(x_+(t-\delta))/\int_{t-\delta}^t f(x_+(s))\,ds < \delta \bar{w}(\delta)$ for $t \geq T_2(\delta).
		$
		The solution to the initial value problem
		\begin{align}\label{eq.ivp}
		y'(t) = \sqrt{\bar{F}(y(t))}, \quad t \geq 0; \quad y(0) = 1,
		\end{align}
		is given by $y(t) = F_U^{-1}(t)$ for $t \geq 0$. Furthermore,
		$
		y''(t) = \frac{1}{2}f(y(t))\mbox{ for } t > 0. 
		$
		Now express the derivatives of $\bar{H}$ in terms of $y$ as follows:
		\[
		\bar{H}'(t)  = K(1+\delta) y'(K(1+\delta)t), \quad \bar{H}''(t) = \frac{K^2 (1+\delta)^2}{2} f(y(K(1+\delta)t)), \quad t \geq T_1(\delta).
		\]
		For short we write $\bar{y}(t) = y(K(1+\delta)t)$ henceforth. Let $x_u$ be the solution to 
		\begin{align*}
		x_u'(t) &= \bar{H}'(t) + \bar{w}(\delta)(1+\delta) \int_{t-\delta}^t f(x_u(s))\,ds, \,\, t \geq T_2(\delta);\quad
		x_u(t) = 1 + x_+(T_2), \,\, t \in [0,T_2].
		\end{align*}
		By construction, $x(t) < x_+(t) < x_u(t)$ for $t \geq T_2(\delta)$. Differentiation yields
		\[
		x_u''(t) = \bar{H}''(t) + \bar{w}(\delta)(1+\delta)\left\{ f(x_u(t)) - f(x_u(t-\delta))  \right\}, \quad t > T_2(\delta).
		\]
		With the notation introduced above, we have
		\begin{align}\label{eq.x_u''_est}
		x_u''(t) &= \frac{K^2 (1+\delta)^2}{2} f(\bar{y}(t)) + \bar{w}(\delta)(1+\delta)f(x_u(t)) \left\{ 1 - \frac{f(x_u(t-\delta))}{f(x_u(t))} \right\}\\ &> \bar{w}(\delta)(1+\delta)f(x_u(t)) \left\{ 1 - \frac{f(x_u(t-\delta))}{f(x_u(t))} \right\}, \quad t > T_2(\delta).\nonumber
		\end{align}
		It is easily demonstrated that $\lim_{t\to\infty}x_u(t-\delta)/x_u(t) = 0$ and hence 
		$f(x_u(t-\delta))/f(x_u(t)) \to 0$ as $t\to\infty$, since $f$ obeys the hypotheses of  Proposition \ref{LEMMA.PRESERVES}. Thus
		\begin{equation}\label{eq.break_cases}
		\liminf_{t\to\infty}\frac{x_u''(t)}{f(x_u(t))} \geq \bar{w}(\delta)(1+\delta), \quad \mbox{for each }\delta > 0,
		\end{equation}
		and there exists $T_3(\delta)  > T_2(\delta)$ such that 
		\begin{equation}\label{eq.est_x_u}
		\frac{x_u''(t)}{f(x_u(t))} > \bar{w}(\delta)(1+\delta)(1-\delta/4)^{1/2}, \quad t \geq T_3(\delta).
		\end{equation}
		Let $C(\delta) = \bar{w}(\delta)(1+\delta)(1-\delta/4)^{1/2}$ and calculate as follows:
		\begin{align*}
		\frac{d}{dt}\left\{ \frac{(x_u'(t))^2}{2} - C(\delta)\bar{F}(x_u(t))  \right\} = \left\{ x_u''(t) - C(\delta) f(x_u(t)) \right\}x_u'(t) > 0, \quad t > T_3(\delta),
		\end{align*}
		using equation \eqref{eq.est_x_u}. Hence,
		$
		(x_u'(t))^2/2 \bar{F}(x_u(t)) > C(\delta)\mbox{ for all } t > T_3(\delta),
		$
		or equivalently
		\[
		\frac{x_u'(t)}{\sqrt{\bar{F}(x_u(t))}} > \sqrt{2C(\delta)}, \quad t > T_3(\delta).
		\]
		Asymptotic integration shows that $
		\liminf_{t\to\infty}F_U(x_u(t))/t \geq \sqrt{2 C(\delta)}
		$.
		Thus
		\begin{equation}\label{eq.lower_est_F_U}
		\frac{F_U(x_u(t))}{t} > \sqrt{2C(\delta)} \sqrt{(1-\delta/4)} = \sqrt{2 \bar{w}(\delta) (1+\delta)(1-\delta/4) }, \quad t> T_4(\delta),
		\end{equation}
		for some $T_4(\delta)> T_3(\delta)$. Therefore
		\[
		\frac{\bar{y}(t)}{x_u(t)} = \frac{F_U^{-1}(K(1+\delta)t) }{ x_u(t) } < \frac{F_U^{-1}(K(1+\delta)t) }{ F_U^{-1}\left( \sqrt{2\bar{w}(\delta)(1+\delta)(1-\delta/4) }\, t \right) }, \quad t \geq T_4(\delta).
		\]
		We require $\delta \in (0,1)$ small enough that
		$
		K(1+\delta) < \sqrt{2\bar{w}(\delta)(1+\delta)(1-\delta/4)}.
		$
		Since $K < \sqrt{2w(0)}$ it is sufficient to choose 
		$
		\delta < \min\left\{(\alpha-1)/(1+\alpha/4),\,1\right\}$, where $\alpha = 2w(0)/K^2>1.
		$
		Hence, by Lemma \ref{lemma.ode_auxiliary}, $\lim_{t\to\infty}\bar{y}(t)/x_u(t) = 0$ for each $\delta > 0$ sufficiently small. By Lemma \ref{lemma.f_preserves_zero}, $\lim_{t\to\infty}f(\bar{y}(t))/f(x_u(t)) = 0$ and combining this limit with \eqref{eq.x_u''_est} shows that 
		$
		\lim_{t\to \infty}x_u''(t)/f(x_u(t)) = \bar{w}(\delta)(1+\delta).
		$
		Thus there exists $T_5(\delta) > T_4(\delta)$ such that 
		\begin{equation}\label{eq.est.final_x_u}
		\bar{w}(\delta) < \frac{x_u''(t)}{f(x_u(t))} < \bar{w}(\delta)(1+\delta)^2, \quad t \geq T_5(\delta).
		\end{equation}
		For each $t \geq T_5(\delta)$, by \eqref{eq.est.final_x_u},
		\[
		\frac{d}{dt}\left\{ \frac{(x_u'(t))^2}{2} - \bar{w}(\delta)(1+\delta)^2 \bar{F}(x_u(t)) \right\}  = \left\{ x_u''(t) - \bar{w}(\delta)(1+\delta)^2 f(x_u(t)) \right\}x_u'(t)< 0.
		\]
		Hence
		$
		x_u'(t) < \sqrt{2\bar{w}(\delta)(1+\delta)^2 \bar{F}_U(x_u(t))} \mbox{ for all } t \geq T_5(\delta).
		$
		Asymptotic integration now readily establishes that
		\[
		\limsup_{t\to\infty}\frac{F_U(x_u(t))}{t} \leq \sqrt{2\bar{w}(\delta)(1+\delta)^2}.
		\] 
		Note that $x(t) < x_u(t)$ for each $t \geq T_2(\delta)$. Therefore, letting $\delta \to 0^+$, we have
		\[
		\limsup_{t\to\infty}\frac{F_U(x(t))}{t} \leq \sqrt{2w(0)}.
		\]
		
		When $K = \sqrt{2w(0)}$, define $\bar{H}(t) = F_U^{-1}(K(1+\delta/2)t)$ for $t\geq T_1(\delta)$. The argument proceeds analogously until equation \eqref{eq.break_cases}. Next use asymptotic integration to show that
		\[
		\liminf_{t\to\infty}\frac{F_U(x_u(t))}{t} \geq \sqrt{2w(0)(1+2\delta)} \quad \mbox{for each }\delta>0.
		\] 
		At this point we must show that for $\delta>0$ sufficiently small 
		\[
		\limsup_{t\to\infty}\frac{\bar{y}(t)}{x_u(t)}\leq \limsup_{t\to\infty}\frac{F_U^{-1}\left(\sqrt{2w(0)}(1+\delta/2)t \right)}{F_U^{-1}\left( \sqrt{2w(0)(1+2\delta)(1-\delta/4)} \right)} = 0
		\]
		by applying Lemma \ref{lemma.ode_auxiliary}. Thus it is sufficient to choose $\delta>0$ such that 
		\[
		\sqrt{2w(0)}(1+\delta/2) < \sqrt{2w(0)(1+2\delta)(1-\delta/4)};
		\]
		this is equivalent to choosing $\delta \in (0,1)$. The proof concludes as in the case $K \in \left(0,\sqrt{2w(0)}\right)$.
		
		Now prove the corresponding limit inferior. Positivity of $H$ yields the trivial lower bound
		$
		x(t) \geq x(0) + \int_0^t W(t-s)f(x(s))\,ds$ for $t \geq 0,
		$
		where $W(t) = \int_0^t w(s) \,ds$. Define $x_-$ by
		\[
		x_-(t) = x(0)/2 + \int_0^t W(t-s)f(x_-(s))\,ds, \quad t \geq 0.
		\]
		By construction, $x_-(t) \leq x(t)$ for $t \geq 0$. Furthermore,
		$
		x_-'(t) = \int_0^t w(t-s)f(x_-(s))\, ds$ for $t \geq 0.
		$
		Apply Theorem \ref{thm.non_expl_rate} to $x_-$ to show that
		$
		\lim_{t\to\infty}F_U(x_-(t))/t = \sqrt{2 w(0)}.
		$
		Therefore
		\[
		\liminf_{t\to\infty}\frac{F_U(x(t))}{t} \geq \lim_{t\to\infty}\frac{F_U(x_-(t))}{t} = \sqrt{2 w(0)},
		\]
		and combining this with the corresponding limit superior establishes that
		\[
		\lim_{t\to\infty}\frac{F_U(x(t))}{t} = \sqrt{2 w(0)}.
		\]		
		Finally, show that $(ii.)$ implies $(i.)$. By positivity, $x(t) > H(t)$ for each $t\geq 0$. Thus, owing to the monotonicity of $F_U$, $F_U(x(t)) \geq F_U(H(t))$ for each $t \geq 0$. Therefore
		\[
		\limsup_{t\to\infty}\frac{F_U(H(t))}{t} \leq \limsup_{t\to\infty}\frac{F_U(x(t))}{t} = \sqrt{2 w(0)},
		\] 
		as required.
	\end{proof}
	\bibliographystyle{abbrv}
	\bibliography{blow_up_refs}

\begin{thebibliography}{10}

\bibitem{appleby2017growth}
J.~A.~D. Appleby and D.~D. Patterson.
\newblock Growth rates of solutions of superlinear ordinary differential
  equations.
\newblock {\em Appl. Math. Lett.}, 71:30--37, 2017.

\bibitem{brunner2012blow}
H.~Brunner and Z.~Yang.
\newblock Blow-up behavior of {H}ammerstein--type {V}olterra integral
  equations.
\newblock {\em J. Integral Equations Appl.}, 24(4):487, 2012.

\bibitem{evtukhov2011asymptotic}
V.~Evtukhov and A.~Samoilenko.
\newblock Asymptotic representations of solutions of nonautonomous ordinary
  differential equations with regularly varying nonlinearities.
\newblock {\em Differ. Equ.}, 47(5):627--649, 2011.

\bibitem{fu2003global}
L.~Fu-Cai, H.~Shu-Xiang, and X.~Chun-Hong.
\newblock Global existence and blow--up of solutions to a nonlocal
  reaction--diffusion system.
\newblock {\em Discrete Contin. Dyn. Syst.}, 9(6):1519--1532, 2003.

\bibitem{GLS}
G.~Gripenberg, S.-O. Londen, and O.~Staffans.
\newblock {\em Volterra Integral and Functional Equations}, volume~34.
\newblock Cambridge University Press, 1990.

\bibitem{kirk2002blow}
C.~Kirk and W.~Olmstead.
\newblock Blow-up in a reactive-diffusive medium with a moving heat source.
\newblock {\em Z. Angew. Math. Phys.}, 53(1):147--159, 2002.

\bibitem{kirk2013system}
C.~Kirk, W.~Olmstead, C.~Roberts, et~al.
\newblock A system of nonlinear {V}olterra equations with blow-up solutions.
\newblock {\em J. Integral Equations Appl.}, 25(3):377--393, 2013.

\bibitem{ma2011blow}
J.~Ma.
\newblock Blow-up solutions of nonlinear {V}olterra integro-differential
  equations.
\newblock {\em Math. Comput. Model.}, 54(11):2551--2559, 2011.

\bibitem{mahmoudi_2017}
N.~Mahmoudi.
\newblock Single--point blow--up for a multi--component reaction--diffusion
  system.
\newblock {\em Discrete Contin. Dyn. Syst.}, 38(1):209--230, 2017.

\bibitem{malolepszy2014blow}
T.~Ma{\l}olepszy.
\newblock Blow-up solutions in one-dimensional diffusion models.
\newblock {\em Nonlinear Anal.}, 95:632--638, 2014.

\bibitem{malolepszy2008blow}
T.~Ma{\l}olepszy and W.~Okrasi{\'n}ski.
\newblock Blow-up conditions for nonlinear {V}olterra integral equations with
  power nonlinearity.
\newblock {\em Appl. Math. Lett.}, 21(3):307--312, 2008.

\bibitem{malolepszy2010blow}
T.~Ma{\l}olepszy and W.~Okrasi{\~n}ski.
\newblock Blow-up time for solutions to some nonlinear {V}olterra integral
  equations.
\newblock {\em J. Math. Anal. Appl.}, 366(1):372--384, 2010.

\bibitem{MYDLARCZYK1994248}
W.~Mydlarczyk.
\newblock A condition for finite blow-up time for a {V}olterra integral
  equation.
\newblock {\em J. Math. Anal. Appl.}, 181(1):248 -- 253, 1994.

\bibitem{mydlarczyk1999blow}
W.~Mydlarczyk.
\newblock The blow-up solutions of integral equations.
\newblock {\em Colloq. Math.}, 79(1):147--156, 1999.

\bibitem{mydlarczyk2005blow}
W.~Mydlarczyk, W.~Okrasi{\'n}ski, and C.~Roberts.
\newblock Blow-up solutions to a system of nonlinear {V}olterra equations.
\newblock {\em J. Math. Anal. Appl.}, 301(1):208--218, 2005.

\bibitem{olmstead1996explosion}
W.~Olmstead and C.~A. Roberts.
\newblock Explosion in a diffusive strip due to a source with local and
  nonlocal features.
\newblock {\em Methods Appl. Anal.}, 3:345--357, 1996.

\bibitem{roberts1997characterizing}
C.~A. Roberts.
\newblock Characterizing the blow-up solutions for nonlinear {V}olterra
  integral equations.
\newblock {\em Nonlinear Anal.}, 30(2):923--933, 1997.

\bibitem{roberts1998analysis}
C.~A. Roberts.
\newblock Analysis of explosion for nonlinear {V}olterra equations.
\newblock {\em J. Comput. Appl. Math.}, 97(1-2):153--166, 1998.

\bibitem{roberts2007recent}
C.~A. Roberts.
\newblock Recent results on blow-up and quenching for nonlinear {V}olterra
  equations.
\newblock {\em J. Comput. Appl. Math.}, 205(2):736--743, 2007.

\bibitem{roberts1996growth}
C.~A. Roberts and W.~Olmstead.
\newblock Growth rates for blow-up solutions of nonlinear {V}olterra equations.
\newblock {\em Quart. Appl. Math.}, 54(1):153--159, 1996.

\end{thebibliography}
\end{document}